\documentclass[a4paper,11pt, twoside]{article}
\usepackage{a4} 
\usepackage{geometry} 
\usepackage{amsmath}
\usepackage{amssymb}
\usepackage{amsthm}
\usepackage{graphicx}
\usepackage{caption,subcaption}
\usepackage{multirow}
\usepackage{xstring}
\usepackage[utf8]{inputenc}
\usepackage{amstext,amsbsy,amsopn,eucal,enumerate}
\usepackage{tikz}
\usepackage{pgfplots}
\usepackage{tabularx}
\usepackage[pagewise]{lineno}
\newcommand{\expn}{\operatorname{e}}

\newcommand{\diag}{\operatorname{diag}}

\newcommand{\beq}{\begin{equation}}
\newcommand{\eeq}{\end{equation}}
\newcommand {\mat}      [1] {\left[\begin{array}{#1}}
\newcommand {\rix}          {\end{array}\right]}
\newcommand {\smat}      [1] {\left[\begin{smallmatrix}{#1}}
\newcommand {\srix}          {\end{smallmatrix}\right]}
\newcommand {\s}      [1] {\begin{smallmatrix}{#1}}
\newcommand {\se}          {\end{smallmatrix}}
\newcommand{\trace}{\operatorname{tr}}

\newtheorem{defn}{Definition}[section]
\newtheorem{remark}[defn]{Remark}
\newtheorem{example}[defn]{Example}
\newtheorem{lem}[defn]{Lemma}
\newtheorem{prop}[defn]{Proposition} 

\newtheorem{thm}[defn]{Theorem}

\title{Model reduction for fully nonlinear stochastic systems}

\author{Martin Redmann\thanks{University of Rostock, Institute of Mathematics, Ulmenstraße 69, 18057 Rostock, Germany, Email: {\tt
martin.redmann@uni-rostock.de}.}
}

\begin{document}

\maketitle

\begin{abstract}
This paper presents a novel model order reduction framework tailored for fully nonlinear stochastic dynamics without  lifting them to quadratic systems and without using linearization techniques. By directly leveraging structural properties of the nonlinearities -- such as local and one-sided Lipschitz continuity or one-sided linear growth conditions -- the approach defines generalized reachability and observability Gramians through Lyapunov-type differential operators. These Gramians enable projection-based reduction while preserving essential dynamics and stochastic characteristics. The paper provides sufficient conditions for the existence of these Gramians, including a Lyapunov-based mean square stability criterion, and derives explicit output error bounds for the reduced order models. Furthermore, the work introduces a balancing and truncation procedure for obtaining reduced systems and demonstrates how dominant subspaces can be identified from the spectrum of the Gramians. The theoretical findings are grounded in rigorous stochastic analysis, extending balanced truncation techniques to a broad class of nonlinear systems under stochastic excitation.
\end{abstract}



\pagestyle{myheadings}
\thispagestyle{plain}
\markboth{M. Redmann}{Model reduction for fully nonlinear stochastic systems}


\section{Introduction}

Model order reduction (MOR) aims to derive low-dimensional approximations of high- or infinite-dimensional dynamical systems, significantly reducing computational complexity while preserving essential features of the original system. A widely used class of MOR techniques relies on projection-based strategies, such as Galerkin or Petrov-Galerkin schemes. In these methods, the key idea is to approximate the solution manifold by a low-dimensional linear subspace and to project the system dynamics onto this subspace to obtain a reduced order model.

There is a rich spectrum of MOR approaches. Proper orthogonal decomposition (POD) \cite{pod} constructs reduced subspaces from data (snapshots) of full-order simulations. Other methods such as the iterative rational Krylov algorithm (IRKA) \cite{irka} rely on interpolation or on optimizing system-theoretic error measures. Balanced truncation (BT) \cite{moo1981}, which is based on reachability and observability Gramians, offers a robust framework with guaranteed error bounds for linear deterministic systems \cite{BT_bound_enns, BT_bound_glover}.

In recent years, substantial progress has been made in adapting MOR strategies to large-scale nonlinear systems. This includes both data-driven approaches \cite{gosea_antoulas, pod_kramer, qian_data} and interpolation- or optimization-based techniques \cite{breiten_benner2, nonlinear_irka}, typically developed for deterministic systems. BT-type extensions to nonlinear systems are, for example, introduced in \cite{nonlinear_bt, Scherpen, verriest_nonlinear_bt}, but none of these approaches came with error bounds. So-called (generalized) incremental BT \cite{incr_bt, gen_incr_bt}, however, uses a new notion of energy functionals leading to a balancing scheme that has nice theoretical properties like stability preservation and an a-priori error bound. In particular, generalized incremental BT \cite{gen_incr_bt} can be linked to this work when $G\equiv 0$, $\Gamma\equiv 0$, $M\equiv 0$ and $h$ is linear in \eqref{original_system}. That ansatz uses quadratic functions to bound the new energy functionals and provides an error bound under symmetry assumptions on the nonlinear term. In this paper, we study balancing in a much more general framework in comparison to \cite{gen_incr_bt}, but we also use quadratic Lyapunov-type functions as a basis and come to similar theoretical findings. Recently, balancing  has been further studied in works such as \cite{morBenG24, Kramer2022}, often by transforming nonlinear dynamics into quadratic-bilinear form via lifting. While lifting expands the applicability of linear MOR tools, it may obscure essential nonlinear properties and complicate stability analysis -- especially in stochastic settings. In fact, all BT techniques for quadratic systems fail to show stability preservation and have no error bounds.

In the stochastic context, MOR becomes even more important. The computational burden increases substantially due to the need for repeated system evaluations, such as in Monte Carlo simulations or uncertainty quantification. Additionally, the complexity of stochastic differential equations (SDEs) is amplified by their connection to high-dimensional partial differential equations via the Feynman-Kac formula. Hence, MOR for SDEs is vital for practical computations. POD-based reduction for SDEs has been studied in \cite{pod_sde}, and several balancing-related or optimization-based MOR methods for linear stochastic systems have been proposed \cite{beckerhartmann, bennerdammcruz, redmannbenner, mliopt}, offering detailed error and stability analyses. However, generalizing these techniques to nonlinear stochastic systems remains challenging. Initial steps toward stochastic bilinear equations were made in \cite{redstochbil} and MOR for SDEs with nonlinear drift \cite{nonlinear_drift} was recently studied. However, these schemes do not extend easily to more general nonlinearities.

The goal of this paper is to advance BT techniques for fully nonlinear stochastic systems with controlled diffusion. Therefore, it can be seen as a natural extension of \cite{redstochbil, nonlinear_drift} as well as of \cite{stoch_nonzero_initial}, where BT was first analyzed for stochastic systems with controlled diffusion. Moreover, our work does not attempt to linearize or rewrite nonlinear dynamics as quadratic systems. The goal is to directly exploit properties like one-sided Lipschitz or linear growth conditions of the nonlinearities. We tailor a balancing scheme including these properties of the system coefficients. Further, the foundation of the proposed MOR technique will be quadratic Lyapunov functions which naturally lead to a notion of algebraic Gramians. Additionally, we clarify under which conditions these algebraic objects exit. Here, we create links to Lyapunov stability for SDEs, see \cite{staboriginal, mao}. On the other hand, we do not need any kind of energy interpretation to justify the usage of these Gramians. In particular, we prove that dominant subspaces of the original nonlinear stochastic dynamics are linked to certain eigenspaces of the underlying Gramians. This leads to a generalized BT method for nonlinear stochastic systems. We explain this procedure that is based on simultaneously diagonalizing both Gramians and subsequently eliminates non-essential components. The main result is an a-priori error bound under a symmetry assumption on the nonlinear vector fields, a requirement that has been pointed out in \cite{gen_incr_bt} in a deterministic framework as well.
\smallskip

The paper is structured as follows. In Section \ref{sec2}, the mathematical framework is introduced. We define a class of controlled stochastic differential equations  incorporating nonlinear drift and diffusion terms. This section further introduces Lyapunov-type quadratic forms and defines a differential operator crucial for analyzing system stability and deriving Gramians.
Section \ref{sec3} develops the theoretical foundations of the proposed MOR method. Generalized reachability and observability Gramians for nonlinear stochastic systems are introduced via Lyapunov-type inequalities. We provide sufficient conditions for the existence of these Gramians, based on one-sided Lipschitz properties and mean square exponential stability.
Finally, links between eigenspaces of the proposed Gramians and state directions that can be truncated with minimal impact on dynamics or output are proved.
Section \ref{sec_BT} describes how to construct a reduced model from the full system by simultaneously balancing the Gramians. A coordinate transformation is used to diagonalize the Gramians, aligning dominant directions. The system is then truncated by retaining only the significant modes. Section \ref{sec_propo_red_sys} contains the most important theoretical finding. We identify  the Gramians of the reduced system and show an a-priori error bound that provides  a criterion for the truncation of the full model.

\section{Setting}\label{sec2}

Let $\left(\Omega, \mathcal F, (\mathcal F_t)_{t\geq 0}, \mathbb P\right)$\footnote{$(\mathcal F_t)_{t\geq 0}$ is right continuous and complete.} be a filtered probability space on which every stochastic process appearing in this paper is defined. Given an $\mathbb R^q$-valued Wiener process $w=\begin{bmatrix}w_1 & \ldots & w_q\end{bmatrix}^\top$, we assume that it is $(\mathcal F_t)_{t\geq 0}$-adapted and its
increments $w(t+h)-w(t)$ are independent of $\mathcal F_t$ for $t, h\geq 0$. Its covariance function is characterized by the matrix $K=\left(k_{ij}\right)_{i, j=1, \ldots, d}$, i.e., $\mathbb E[w(t)w(t)^\top]=K t$ holds.
We consider a high-dimensional nonlinear stochastic system of the following form:
 \begin{subequations}\label{original_system}
\begin{align}\label{stochstatenew}
             dx(t)&=[f(x(t))+Bu(t)+ G(x(t))u(t)]dt+[\Gamma(x(t))+M(u(t))]dw(t),\\ \label{output_eq}
            y(t) &= h(x(t))+E u(t),\quad t\geq 0,
\end{align}
\end{subequations}
with  $G(x):=\begin{bmatrix}g_1(x) & \ldots & g_m(x)\end{bmatrix}$ and $\Gamma(x):=\begin{bmatrix}\gamma_1(x) & \ldots & \gamma_q(x)\end{bmatrix}$, where $f, g_i, \gamma_j: \mathbb R^n\to\mathbb R^{n}$ and $h:\mathbb R^n\to \mathbb R^p$ are (nonlinear) functions that are sufficiently nice in order to guarantee a unique global solution (e.g. local Lipschitz continuity combined with a growth condition). We specify the particular requirements later. Moreover, $M(u):=\begin{bmatrix}M_1 u & \ldots & M_q u\end{bmatrix}$ is a linear map on $\mathbb R^m$ defined by matrices $M_j\in \mathbb R^{n\times m}$. The additional control parts are determined by $B\in \mathbb R^{n\times m}$ and $E\in \mathbb R^{p\times m}$. We assume without loss of generality that $f(0)=g_i(0)= \gamma_j(0)=0$ and $h(0)=0$ for all $i\in\{1, \dots, m\}$ and $j\in\{1, \dots, q\}$. If this is originally not true, then the initial values of these nonlinearities can be made part of the control terms $Bu$, $M_ju$ and $Eu$. In addition, $g_i$ shall either be identically zero or non-constant in order to distinguish $G(\cdot) u$ from $Bu$. The stochastic control $u\in L^2_T$ takes values in $\mathbb R^m$, is $(\mathcal F_t)_{t\geq 0}$-adapted  and satisfies $\|u\|_{L^2_T}^2:=\mathbb E \int_0^T \|u(t)\|^2 dt<\infty$ for any $T>0$, where $\|\cdot\|$ denotes the Euclidean norm and $\langle \cdot, \cdot \rangle$ the corresponding inner product. If $u$ is additionally square integrable on the entire time line $[0, \infty)$, we write $u\in L^2$.

We introduce Lyapunov-type quadratic forms $V_X(x):= x^\top X x$, where the dimension of $x$ is generic and depends on the number of rows/columns of the positive (semi)definite matrix $X$. Let us further denote the gradient by $D$ and Hessian by $D^2$. We introduce a differential operator $L$ that is essential for the proposed MOR procedure. It will be used to define the general concept of Gramians and is given by
\begin{equation}\label{layp_op}
\begin{aligned}
LV(x, y; \delta)&=\langle D V(x-y), f(x)-f(y)\rangle \\  &\quad+ \frac{\delta^2}{2}  \trace\left((G(x)-G(y)) (G(x)-G(y))^\top D^2 V(x-y)\right)\\  &\quad+ \frac{1}{2}  \trace\left((\Gamma(x)-\Gamma(y)) K (\Gamma(x)-\Gamma(y))^\top D^2 V(x-y)\right)
\end{aligned}
\end{equation}
for $x, y\in\mathbb R^n$, a twice differentiable function $V:\mathbb R^n\to \mathbb R$ and $\delta >0$. If we replace $f$, $G$ and $\Gamma$ in \eqref{layp_op} by functions defined on $\mathbb R^r$, $r\neq n$, we assume $x, y\in\mathbb R^r$ and define $V$ as a function on $\mathbb R^r$.
\begin{prop}\label{prop_lyap}
Setting $V(x)=V_X(x)= x^\top X x$ in \eqref{layp_op} for a positive semidefinite matrix $X\in\mathbb R^{n\times n}$, we obtain
 \begin{align*}
LV_X(x, y; \delta)
&=2\langle X(x-y), f(x)-f(y)\rangle + \delta^2\trace\left((G(x)-G(y)) (G(x)-G(y))^\top X\right)\\ &\quad+ \trace\left((\Gamma(x)-\Gamma(y)) K (\Gamma(x)-\Gamma(y))^\top X\right)\\
&=2\langle x-y, X(f(x)-f(y))\rangle + \delta^2\|X^{\frac{1}{2}}(G(x)-G(y))\|_F^2+ \|X^{\frac{1}{2}}(\Gamma(x)-\Gamma(y)) K^{\frac{1}{2}}\|_F^2.
\end{align*}
\end{prop}
\begin{proof}
 We insert $D V_{X}(x)=2 X x$ and $D^2 V_{X}(x)=2 X$ into \eqref{layp_op} and obtain the second identity using the existence of the square root $X^{\frac{1}{2}}$, the properties of the trace and the definition of the Frobenius norm $\|\cdot\|_F$.
\end{proof}

\section{Gramians -- definition, existence criteria, and dominant subspace characterization}\label{sec3}

The section deals with algebraic Gramians,  the central objects required for the MOR procedure proposed in this paper. We begin  by introducing Gramians based on the Lyapunov operator $L$ defined in \eqref{layp_op}. Secondly, we discuss a Lyapunov criterion for exponential mean square asymptotic stability as well as a more general notion of one-sided Lipschitz continuity with negative constants. These conditions yield the existence of  the Gramians of interest. Finally, we analyze the relation between dominant subspaces of \eqref{original_system} and the eigenspaces of the proposed Gramians.

\subsection{Definition of Gramians}

We introduce and analyze two objects $U_X$ and $S_X$ for a positive (semi)definite matrix $X$. We need those for defining a reachability Gramian.
\begin{prop}\label{prop_U_S}
Let $z\in\mathbb R^m$ and $X$ be a positive (semi)definite matrix. Defining   \begin{align}\label{def_U}
 U_{X} = \mathcal U - \sum_{i, j=1}^q M_i^\top X M_j k_{ij}
      \end{align}
with the positive definite matrix $\mathcal U$ being a free parameter, we have that \begin{align}\label{traceU_rel}
 z^\top U_{X} z =  z^\top \mathcal U z -   \trace\left(M(z) K M(z)^\top X\right).                                                                                \end{align}
Moreover, introducing \begin{align}\label{def_S}
    S_{X}(x,y)&=B^\top X(x-y)+\sum_{i, j=1}^q M_i^\top X (\gamma_j(x)-\gamma_j(y))k_{ij}
\end{align}
for $x, y \in\mathbb R^n$, we obtain
\begin{align} \label{traceS_rel}
 z^\top  S_{X}(x, y) = z^\top B^\top X (x-y) + \trace\left((\Gamma(x) - \Gamma(y)) K M(z)^\top X\right)
\end{align}
\end{prop}
\begin{proof}
  Let $(e_i)$ be the canonical basis of $\mathbb R^q$. Hence, we have
 \begin{align*}
  \trace\left((\Gamma(x)- \Gamma(y))K M(z)^\top X\right)&=\sum_{i=1}^q e_i^\top M(z)^\top X(\Gamma(x)- \Gamma(y)) K e_i \\
  &= z^\top \sum_{i, j=1}^q M_i^\top X (\gamma_j(x)- \gamma_j(y)) k_{ij}\\
  &= z^\top  S_{X}(x, y) - z^\top B^\top X(x-y)
 \end{align*}
 exploiting the definition of $S_{X}$ and the fact that $k_{ij}=k_{ji}$. Further, by the definition of $U_X$, we obtain  \begin{align*}
     \trace\left(M(z) K M(z)^\top X\right)&=
     \sum_{i=1}^q e_i^\top M(z)^\top X M(z) K e_i = z^\top \sum_{i, j=1}^q M_i^\top X M_j k_{ij} z \\
     &= - z^\top U_{X} z+ z^\top \mathcal U z.
                            \end{align*}
\end{proof}
Using the objects introduced in Proposition \ref{prop_U_S}, we define reachability Gramians in the following.
\begin{defn}\label{def_P}
 A positive definite matrix $P$ is called (full) reachability Gramian
  if the matrix  $U_{P^{-1}}$ defined in  \eqref{def_U} is positive definite
 and if \begin{align}\label{inequalityP}
 LV_{P^{-1}}(x, y; \delta)\leq -S_{P^{-1}}(x,y)^\top U_{P^{-1}}^{-1} S_{P^{-1}}(x, y)
                      \end{align}
holds for all $x, y\in\mathbb R^n$ and some $\delta>0$.
The right-hand side in \eqref{inequalityP} is defined by \eqref{def_U}  and \eqref{def_S}. The  left-hand side is given by Proposition \ref{prop_lyap}. If we replace the general $y\in \mathbb R^n$ in \eqref{inequalityP} by $y=0$, we call $P$ simplified reachability Gramian.
\end{defn}
The role of a simplified reachability Gramian becomes clear when characterizing dominant subspaces of \eqref{original_system} in Section \ref{sec_dom_sub} while a full Gramian $P$  reveals its relevance in the later error analysis, see Section \ref{sec_propo_red_sys}.
\begin{defn}\label{def_Q}
 A positive (semi)definite matrix $Q$ is called (full) observability Gramian if  \begin{align}\label{inequalityQ}
 LV_{Q}(x, y; \delta)\leq -(h(x)- h(y))^\top (h(x)- h(y))
                      \end{align}
holds for all $x, y\in\mathbb R^n$ and some $\delta>0$, where the left-hand side in \eqref{inequalityQ} is given by Proposition \ref{prop_lyap}. If we replace the general $y\in \mathbb R^n$ in \eqref{inequalityQ} by $y=0$, we call $Q$ simplified observability Gramian.
\end{defn}
Both a full and a simplified observability Gramian can be motivated in the context of dominant subspace detection in Section \ref{sec_dom_sub}. However, a full Gramian $Q$ is essential when proving the error bound in Section \ref{sec_propo_red_sys}. It is important to mention that inequalities like \eqref{inequalityP} and \eqref{inequalityQ} can be found in \cite{gen_incr_bt}, where generalized incremental BT was studied in a framework, where $h$ is linear and $G\equiv 0$, $\Gamma\equiv0$ as well as $M\equiv 0$. In \cite{gen_incr_bt}, it was further pointed out that point symmetry of the nonlinearities is crucial in the context of the reachability Gramian. This symmetry will later play a role when we prove an error bound for the proposed MOR scheme.
\begin{remark}\label{remark_lin_case}
Let us focus on the special case of having a linear system, i.e., $f(x)= Ax$, $G\equiv 0$, $\gamma_j(x)=N_jx$ and $h(x)=Cx$ for matrices $A, N_j\in\mathbb R^{n\times n}$ and $C\in\mathbb R^{p\times n}$. Inserting this into the identity given in Proposition \ref{prop_lyap}, we find that
\begin{align*}
LV_X(x, y; \delta)
&=2\langle X(x-y), A(x-y)\rangle + \trace\left(\Gamma(x-y) K \Gamma(x-y)^\top X\right)    \\
&= (x-y)^\top (A^\top X+ X A)(x-y)+\trace\left(\sum_{j, i=1}^q N_j(x-y) k_{ji} (x-y)^\top N_i^\top X\right)\\
&= (x-y)^\top (A^\top X+ X A+\sum_{i, j=1}^q  N_i^\top X N_j k_{ij})(x-y)\end{align*}
exploiting that $k_{ij}=k_{ji}$. Inserting the linear structure of $h$ into \eqref{inequalityQ}, we obtain the matrix inequality \begin{align}\label{Q_mat_ineq}
 A^\top Q+ Q A+\sum_{i, j=1}^q  N_i^\top Q N_j k_{ij}\leq -C^\top C.                                                                                                                                                                                                                                                              \end{align}
This means that a solution of \eqref{Q_mat_ineq} is an observability Gramian. This observability Gramian is frequently used for BT-related MOR applied to stochastic linear systems \cite{bennerdammcruz, redmannbenner}.  Moreover, observing that   $S_{P^{-1}}(x,y)=(B^\top P^{-1}+\sum_{i, j=1}^q M_i^\top P^{-1} N_j k_{ij})(x-y)$ in this special case, we obtain that a solution of \begin{align*}
 & A^\top P^{-1} + P^{-1}  A+\sum_{i, j=1}^q  N_i^\top P^{-1}  N_j k_{ij}\\
 &\leq -\Big(P^{-1} B+\sum_{i, j=1}^q N_i^\top P^{-1} M_j k_{ij}\Big) U_{P^{-1}}^{-1}\Big(B^\top P^{-1}+\sum_{i, j=1}^q M_i^\top P^{-1} N_j k_{ij}\Big)                                                                                                                                                                                                                                                                                                                                                                                                                                                                                        \end{align*}
satisfies \eqref{inequalityP}, where $U_{P^{-1}}= \mathcal U - \sum_{i, j=1}^q M_i^\top P^{-1} M_j k_{ij}$. This is a Gramian used in \cite{stoch_nonzero_initial} or in \cite{bennerdammcruz}, where the  diffusion is additionally uncontrolled. Setting $N_i=0$, $M_i=0$ and $\mathcal U=I$ in this matrix inequality, we obtain  $A P + P  A^\top\leq -B B^\top$ yielding the typical Gramian  for deterministic linear systems. Let us further notice that the concept of a simplified Gramian is redundant in the linear case as it also satisfies the criterion for the full Gramian.
\end{remark}

\subsection{Existence criteria for (simplified)  Gramians}

In this section, we establish sufficient criteria for the existence of reachability and observability Gramians. In fact, the existence of simplified Gramian can be linked to Lyapunov criteria for asymptotic exponential stability, whereas the full Gramians exist under a one-sided Lipschitz property with a negative constant.\smallskip

First, we need to introduce the control vector
$\tilde u:= \begin{bmatrix}\tilde u_1 & \ldots & \tilde u_m\end{bmatrix}^\top$, where
\begin{align}\label{multipl_control}
\tilde u_i\equiv \begin{cases}0&\text{ if } g_i\equiv 0,\\
  u_i&\text{ else.}
  \end{cases}
\end{align}
This function $\tilde u$ only contains the controls that enter the multiplicative part in the drift. It is often important to distinguish between the additive and the multiplicative controls like in the following stability result. It contains a Lyapunov criterion for mean square asymptotic stability that ensures existence of the simplified Gramians.
\begin{thm}\label{thm_global_stab}
Suppose that $M_j=B=0$ in \eqref{stochstatenew} for all $j\in \{1, \dots, q\}$. We assume that the multiplicative controls, defined in \eqref{multipl_control}, are deterministic and satisfy $\|\tilde u\|_{L^2}^2:= \int_0^\infty \|\tilde u(t)\|^2 dt<\infty$. Further, let $X>0$ be a positive definite matrix, $V_X(x)=x^\top X x$ and $LV_{X}$ given by Proposition \ref{prop_lyap}. If we have
\begin{align}\label{stab_Lyap}
LV_{X}(x, 0; \delta)\leq -\lambda  V_{X}(x)                                                                                                                                                                                                                                                                                                                                                                                                                                                                                                                                                                                                                                                                                                        \end{align}
for some $\lambda>0$, all $x\in \mathbb R^n$ and some $\delta>0$. Then, there is a constant $k>0$, so that
\begin{align*}
\mathbb E \left\|x(t)\right\|^2\leq k \expn^{(\|\tilde u\|_{L^2}/\delta)^2 }\left\|x_0\right\|^2  \expn^{-\lambda t}
\end{align*}
for all $t\geq 0$ and initial values $x_0\in\mathbb R^n$.
\end{thm}
\begin{proof}
 We apply Lemma \ref{lemstochdiff} with $a(t)=f(x(t))+G(x(t))u(t)$, $B(t)= \Gamma(x(t))$ and $V(x)=V_X(x)=x^\top X x$. Hence, we  obtain \begin{align*}
 dV_X(x(t))=&\left[2\left\langle X x(t), f(x(t))+G(x(t))u(t)\right\rangle +  \trace\left(\Gamma(x(t)) K \Gamma(x(t))^\top X\right) \right] dt\\
 &+    2\left\langle X x(t), \Gamma(x(t))dw(t)\right\rangle.                                                                                                                      \end{align*}
 We take the expectation and use that the Ito integral has mean zero. This results in \begin{align*}
 \frac{d}{dt}\mathbb E[V_X(x(t))]=\mathbb E\left[2\left\langle X x(t), f(x(t))+G(x(t))u(t)\right\rangle +  \trace\left(\Gamma(x(t)) K \Gamma(x(t))^\top X\right) \right]                                                                                                                     \end{align*}
 for almost all $t\geq 0$.  We exploit that \begin{align}\nonumber
   2\left\langle X x(t), G(x(t))u(t)\right\rangle &= 2\left\langle X^{\frac{1}{2}} x(t), X^{\frac{1}{2}} G(x(t))\tilde u(t)\right\rangle
   =2\trace\Big(X^{\frac{1}{2}} \delta G(x(t))(\tilde u(t)/\delta) x(t)^\top X^{\frac{1}{2}}\Big)\\ \nonumber
   &\leq  \delta^2\left\|X^{\frac{1}{2}} G(x(t))\right\|^2_F + \left\|(\tilde u(t)/\delta) x(t)^\top X^{\frac{1}{2}}\right\|^2_F\\ \label{estimate_on_G}
   &= \delta^2\left\|X^{\frac{1}{2}} G(x(t))\right\|^2_F + V_{X}(x(t)) (\left\|\tilde u(t)\right\|/\delta)^2.
                 \end{align}
Therefore, we obtain by Proposition \ref{prop_lyap} and \eqref{stab_Lyap} that \begin{align*}
 \frac{d}{dt}\mathbb E[V_X(x(t))]&\leq\mathbb E\left[LV_{X}(x(t), 0; \delta) \right] + \mathbb E[V_{X}(x(t))] (\left\|\tilde u(t)\right\|/\delta)^2\\
 &\leq \Big((\left\|\tilde u(t)\right\|/\delta)^2 -\lambda\Big)\mathbb E[V_{X}(x(t))]                                                                                                                   \end{align*}
 for almost all $t\geq 0$. Exploiting Lemma \ref{gron_dif} with $\beta(t)=(\left\|\tilde u(t)\right\|)/\delta)^2 -\lambda$, we find that $\mathbb E[V_X(x(t))] \leq V_X(x_0) \expn^{\int_0^t(\left\|\tilde u(s)\right\|/\delta)^2 -\lambda ds}$ for all $t\geq 0$. As ${V_X(\cdot)}^{\frac{1}{2}}=\|X^\frac{1}{2} \cdot\|$ is a norm, it is equivalent to the Euclidean norm $\left\|\cdot\right\|$ and therefore the result follows.
\end{proof}
If $G\equiv 0$ in Theorem \ref{thm_global_stab}, then the result is covered by the classical Lyapunov stability theory for (uncontrolled) stochastic differential equations, see \cite{staboriginal, mao}.
\begin{remark}
 Looking at Proposition \ref{prop_lyap}, we can write $LV_{X}(x, 0; \delta)=LV_{X}(x, 0; 0)+  \delta^2\|X^{\frac{1}{2}}G(x)\|_F^2$. Now, if there is a $\tilde \lambda>0$, so that $LV_{X}(x, 0; 0)\leq -\tilde \lambda  V_{X}(x)$ for all $x\in \mathbb R^n$ (implying mean square exponential stability for \eqref{stochstatenew} with $B=M_j=0$ and $G\equiv 0$ \cite{mao}) and $G$ satisfies a linear growth condition in some norm, i.e., $\|X^{\frac{1}{2}}G(x)\|_F\leq c \|X^{\frac{1}{2}}x\|$ for some suitable constant $c>0$ and all $x\in\mathbb R^n$.  Then, \begin{align*}
  LV_{X}(x, 0; \delta)\leq -\tilde \lambda  V_{X}(x)+  c^2 \delta^2\|X^{\frac{1}{2}}x\|^2 = -\lambda  V_{X}(x)                                                                                                                                                                                                                                                              \end{align*}
for all $x\in\mathbb R^n$ and $\lambda = \tilde \lambda-c^2\delta^2$. As $\lambda>0$ for a sufficiently small $\delta$,  the Lyapunov stability criterion \eqref{stab_Lyap} is satisfied. On the other hand, $G$ does not necessarily need to satisfy such a growth condition, as the following example demonstrates. Let $\circ$ denote component-wise multiplications/powers. We set $f(x)=Ax-x^{\circ 3}$, $m=q=K=1$, $G(x)=\Gamma(x)=x^{\circ 2}$ and $X=I$. Suppose further that $2\left\langle x, Ax\right\rangle\leq -\tilde \lambda \|x\|^2$ for all $x\in \mathbb R^n$ and some $\tilde \lambda >0$ (implying the Hurwitz property of $A$). Then, we have $ LV_{I}(x, 0; 0)  = 2\left\langle x, Ax\right\rangle -2 \sum_{i=1}^n x_i^4 + \sum_{i=1}^n x_i^4 \leq -\tilde \lambda  V_{I}(x)$ and \begin{align*}
   LV_{I}(x, 0; \delta)  = 2\left\langle x, Ax\right\rangle -2 \sum_{i=1}^n x_i^4 +(1+\delta^2) \sum_{i=1}^n x_i^4\leq  -\tilde \lambda  V_{I}(x)                                                                                                                                                                                                                                                              \end{align*}
for all $x\in \mathbb R^n$ if $0<\delta\leq 1$.
\end{remark}
A more general requirement than the Lyapunov criterion of Theorem \ref{thm_global_stab} is now used to ensure the existence of the full Gramians. In \cite{gen_incr_bt} (deterministic case with $G\equiv 0$ and $h$ linear) such a requirement is also used to ensure the existence of Gramians. There, it is called quadratic incremental stability.
\begin{thm}\label{thm_gram_exist}
Let us assume that the expression defined by Proposition \ref{prop_lyap} satisfies \begin{align}\label{exist_gram_prop}
LV_{X}(x, y; \delta)\leq -\lambda  V_{X}(x-y)                                                                                                                                                                                                                                                                                                                                                                                                                                                                                                                                                                                                                                                                                                        \end{align}
for all $x, y\in \mathbb R^n$, a positive definite matrix $X$, $\lambda>0$ and some $\delta>0$.
\begin{itemize}
 \item
A Gramian $P$ in Definition \ref{def_P} exists if $\Gamma$ is globally Lipschitz continuous in some norm, i.e., there exists a constant $c>0$ such that \begin{align}\label{lip_gamma}
 \|X^{\frac{1}{2}}(\Gamma(x)-\Gamma(y))\|_F\leq c \|X^{\frac{1}{2}}(x-y)\|                                                                                                                                                                                                                                                                                                                                                                                                                                                                                                     \end{align}                                                                                                                                                                                                                                                 for all $x, y\in \mathbb R^n$. In case \eqref{exist_gram_prop} and \eqref{lip_gamma} are true for $y=0$, i.e., \eqref{stab_Lyap} holds and $\Gamma$ satisfies a linear growth condition, then a simplified Gramian $P$ is well-defined.
 \item
A Gramian $Q$ according to Definition \ref{def_Q} exist if
\begin{align}\label{lip_h}
 \|X^{\frac{1}{2}}(h(x)-h(y))\|\leq c \|X^{\frac{1}{2}}(x-y)\|                                                                                                                                                                                                                                                                                                                                                                                                                                                                                                   \end{align}                                                                                                                                                                                                                                                 for all $x, y\in \mathbb R^n$. If \eqref{exist_gram_prop} and \eqref{lip_h} are satisfied for $y=0$, then a simplified Gramian $Q$ is well-defined.
\end{itemize}
\end{thm}
\begin{proof}
 We define $P=(\epsilon X)^{-1}$ for a suitable $\epsilon>0$ and multiply both sides of \eqref{exist_gram_prop} with $\epsilon$ resulting in $LV_{P^{-1}}(x, y; \delta)\leq -\lambda  V_{P^{-1}}(x-y)$. Let $c_{\min}>0$ be the smallest eigenvalue of $\mathcal{U}$ introduced in Proposition \ref{prop_U_S} and let $\mu_{\max}\geq 0$ be the  largest eigenvalue of $\sum_{i, j=1}^q M_i^\top X M_j k_{ij}$. Then, $c_{\min}-\epsilon \mu_{\max}$ is a lower bound for all eigenvalues of $U_{P^{-1}}= \mathcal U - \epsilon\sum_{i, j=1}^q M_i^\top X M_j k_{ij}$. Hence, $U_{P^{-1}}$ is positive definite, if we choose an $\epsilon>0$ ensuring $c_{\min}-\epsilon \mu_{\max}>0$. Now, the largest eigenvalue of $U_{P^{-1}}^{-1}$ is bounded from above by $\frac{1}{c_{\min}-\epsilon \mu_{\max}}$                                                                                                                                                                                                                                                                                                                                                                                                                                                                                                                                                                                                                                                                                                                                                                                                                                                                 leading to
 \begin{align}\label{some_est}
 -S_{P^{-1}}(x, y)^\top  U_{P^{-1}}^{-1} S_{P^{-1}}(x, y)&\geq - S_{P^{-1}}(x, y)^\top S_{P^{-1}}(x, y)  \frac{1}{c_{\min}-\epsilon \mu_{\max}}.                                                                                                                                                                                                                                                                                                                                                                                                                                                                                                                                                                                                                                                                                                                                                                                                                                                                                           \end{align}
Now, it holds that \begin{align*}
  S_{P^{-1}}(x, y)^\top S_{P^{-1}}(x, y) &=  \|B^\top P^{-1} (x-y)+ \sum_{i, j=1}^q  M_i^\top P^{-1} (\gamma_j(x)-\gamma_j(y)) k_{ij}\|^2\\
  &\leq 2\epsilon^2\Big(\|B^\top X (x-y)\|^2+ \|\sum_{i, j=1}^q  M_i^\top X (\gamma_j(x)-\gamma_j(y)) k_{ij}\|^2\Big).
               \end{align*}
We obtain that $\|B^\top X (x-y)\|\leq c_B \|X^{\frac{1}{2}} (x-y)\|$ for some $c_B>0$ and \begin{align*}                                                                                            \|\sum_{i, j=1}^q  M_i^\top X (\gamma_j(x)-\gamma_j(y)) k_{ij}\| &\leq                                                                                            \sum_{i, j=1}^q  \|{X^\frac{1}{2}}M_i\| \vert k_{ij}\vert \|X^{\frac{1}{2}} (\gamma_j(x)-\gamma_j(y))\|\\
& \leq                                                                                            c_M\|X^{\frac{1}{2}}(x-y)\|
\end{align*}
for a suitable constant $c_M>0$ applying \eqref{lip_gamma}. Therefore, we have \begin{align*}
  S_{P^{-1}}(x, y)^\top S_{P^{-1}}(x, y) \leq 2\epsilon^2(c_B^2+c_M^2)\|X^{\frac{1}{2}}(x-y)\|^2 = 2\epsilon(c_B^2+c_M^2) V_{P^{-1}}(x-y).
               \end{align*}
Inserting this into \eqref{some_est} yields \begin{align*}
 -S_{P^{-1}}(x, y)^\top  U_{P^{-1}}^{-1} S_{P^{-1}}(x, y) \geq - \epsilon \frac{2(c_B^2+c_M^2)}{c_{\min}-\epsilon \mu_{\max}}V_{P^{-1}}(x-y)\geq -\lambda V_{P^{-1}}(x-y)                                                                                                                                                                                                                                                                                                                                                                                                                                                                                                                                                                                                                                                                                                                                                                                                                                                                                         \end{align*}
 given a sufficiently small $\epsilon>0$. This proves that $P=(\epsilon X)^{-1}$ is  reachability Gramian for $\epsilon$ being small enough. Repeating the same steps for $y=0$, we see that the simplified version of the Gramian $P$ exists. For the second part of the proof, let us define $Q=\epsilon^{-1} X$. Now, we multiply both sides of \eqref{exist_gram_prop} with $\epsilon^{-1}$ yielding $LV_{Q}(x, y; \delta)\leq -\lambda  V_{Q}(x-y)$. We obtain that \begin{align*}
 -(h(x)- h(y))^\top (h(x)- h(y)) &= - \|h(x)- h(y)\|^2 \geq -c_X^2 \|X^{\frac{1}{2}}(x-y)\|^2\\
 &=  -\epsilon c_X^2 \|Q^{\frac{1}{2}}(x-y)\|^2 = -\epsilon c_X^2 V_{Q}(x-y)
                      \end{align*}
for some $c_X>0$ and all $x, y\in\mathbb R^n$ combining \eqref{lip_h} with the equivalence of the norms. Consequently, we find that $LV_{Q}(x, y; \delta)\leq  -(h(x)- h(y))^\top (h(x)- h(y))$ if $\epsilon$ is chosen, such that $\epsilon c_X^2\leq \lambda$. Therefore, $Q=\epsilon^{-1} X$ is an observability Gramian if $\epsilon$ is small enough. Repeating the same steps for the case of $y=0$, we see that the simplified observability Gramian is well-defined. This concludes the proof.
\end{proof}
\begin{remark}
Condition \eqref{exist_gram_prop} is a one-sided Lipschitz property of the coefficients $(f, G, \Gamma)$. Theorem \ref{thm_gram_exist} tells us that the Gramians $P$ and $Q$ exist if we find an inner product (and a corresponding norm), in which the one-sided Lipschitz constant is negative  additionally given that $\Gamma$ and $h$ are globally Lipschitz. For the simplified Gramian case, the one-sided Lipschitz property becomes a Lyapunov stability criterion \eqref{stab_Lyap} while $\Gamma$ and $h$ only have to satisfy a linear growth condition.
\end{remark}

\subsection{Dominant subspace detection using Gramians}\label{sec_dom_sub}

\subsubsection{Directions with low contributions to the state vector}\label{subsecP}

If $(p_k)$ is an ONB of $\mathbb R^n$, then the state variable can be written as \begin{align}\label{eigen_rep}
x(t)= \sum_{k=1}^n \langle x(t), p_{k}
\rangle \,p_k.                                                                                                                                                                                                                                                                                                                                                                                                                                                                                                                                                                                                                                                                                                                                                                                                                                                                                                                                                                                                                                                        \end{align}
Our goal is to choose a particular basis $(p_k)$ associated to the (simplified) reachability Gramian $P$ introduced in Definition \ref{def_P} and identify those directions $p_k$ that barely contribute to $x$. This is the basis for a truncation procedure leading to a reduced order model.
\begin{thm}\label{energy_est}
Let $P$ be a simplified reachability Gramian and $\mathcal U$ be the positive definite matrix introduced in Definition \ref{def_P}. Further,  $x$ denotes the solution of \eqref{stochstatenew}. Let $(p_k, \lambda_{k})$ be a basis of eigenvectors of $P$, yielding \eqref{eigen_rep}, equipped with the eigenvalues $\lambda_{k}=\lambda_k(\delta)>0$, $\delta>0$. Assuming that $x_0=0$ and that $\tilde u$, defined in \eqref{multipl_control}, is deterministic, we have
\begin{align*}
\sup_{t\in[0, T]}\mathbb E \langle x(t), p_{k}  \rangle^2 &\leq \lambda_{k} \exp{\left\{\|\tilde u\|_{L^2_T}^2/\delta^2 \right\}}\|\mathcal U^{\frac{1}{2}}  u\|_{L^2_T}^2.
\end{align*}
\end{thm}
\begin{proof}
We exploit Lemma \ref{lemstochdiff} with $a(t)=f(x(t))+Bu(t)+G(x(t))u(t)$, $B(t)= \Gamma(x(t))+M(u(t))$ and $V(x)=V_{P^{-1}}(x)=x^\top P^{-1} x$. Taking the expectation and calculating $D V_{P^{-1}}(x)=2 P^{-1}x$ as well as $D^2 V_{P^{-1}}(x)=2 P^{-1}$, we  obtain \begin{align}\label{eq_after_ito_exp}
 \frac{d}{dt}\mathbb E[V_{P^{-1}}(x(t))]=&\mathbb E\left[2\left\langle P^{-1} x(t), f(x(t))+Bu(t)+G(x(t))u(t)\right\rangle\right]\\ \nonumber
 &+  \mathbb E\left[\trace\left((\Gamma(x(t))+M(u(t))) K (\Gamma(x(t))+M(u(t)))^\top P^{-1}\right) \right].                                                                                                                    \end{align}
We establish the same estimate like in \eqref{estimate_on_G} and hence have \begin{align}\label{est_bil_term}
   2\left\langle P^{-1} x(t), G(x(t))u(t)\right\rangle\leq \delta^2\left\|P^{-\frac{1}{2}} G(x(t))\right\|^2_F + V_{P^{-1}}(x(t)) (\left\|\tilde u(t)\right\|/\delta)^2.
                 \end{align}
We integrate \eqref{eq_after_ito_exp} over $[0, t]$, use inequality \eqref{est_bil_term} and Proposition \ref{prop_lyap} yielding
\begin{align}\label{first_ernergy_est}
 \mathbb E[V_{P^{-1}}(x(t))]&\leq  \int_0^t \mathbb E\left[LV_{P^{-1}}(x(s), 0; \delta)\right]+2  \mathbb E\left[\left\langle P^{-1} x(s), Bu(s)\right\rangle \right]\\
 \nonumber &\quad \quad+ \mathbb E\left[2 \trace\left(\Gamma(x(s)) K M(u(s))^\top P^{-1}\right)+ \trace\left(M(u(s)) K M(u(s))^\top P^{-1}\right)\right]ds\\ \nonumber
 &\quad  + \int_0^t \mathbb E[V_{P^{-1}}(x(s))] (\left\|\tilde u(s)\right\|/\delta)^2 ds.
 \end{align}
We insert \eqref{traceU_rel} and \eqref{traceS_rel} with $z=u(s)$, $x=x(s)$, $y=0$ and $X=P^{-1}$ into \eqref{first_ernergy_est}. This yields
 \begin{align*}
\mathbb E[V_{P^{-1}}(x(t))] &\leq  \int_0^t \mathbb E\left[LV_{P^{-1}}(x(s), 0; \delta)\right]+2 \mathbb E\left[u(s)^\top S_{P^{-1}}(x(s), 0)  \right]\\
 &\quad \quad + \mathbb E\left[u(s)^\top \mathcal U u(s)\right]-\mathbb E\left[u(s)^\top U_{P^{-1}} u(s)\right]ds\\
 &\quad  + \int_0^t \mathbb E[V_{P^{-1}}(x(s))] (\left\|\tilde u(s)\right\|/\delta)^2 ds.
 \end{align*}
 We insert \eqref{inequalityP} into the above inequality and obtain
 \begin{align*}
\mathbb E[V_{P^{-1}}(x(t))] &\leq  \int_0^t \mathbb E\left[-S_{P^{-1}}(x(s), 0)^\top U_{P^{-1}}^{-1} S_{P^{-1}}(x(s),0)\right]+2 \mathbb E\left[u(s)^\top S_{P^{-1}}(x(s), 0)  \right]\\
 &\quad \quad + \mathbb E\left[u(s)^\top \mathcal U u(s)\right]-\mathbb E\left[u(s)^\top U_{P^{-1}} u(s)\right]ds\\
 &\quad  + \int_0^t \mathbb E[V_{P^{-1}}(x(s))] (\left\|\tilde u(s)\right\|/\delta)^2 ds\\
 &=\|\mathcal U^{\frac{1}{2}} u\|_{L^2_t}^2- \mathbb E \int_0^t \|U_{P^{-1}}^{-\frac{1}{2}} S_{P^{-1}}(x(s), 0)-U_{P^{-1}}^\frac{1}{2} u(s)\|^2 ds\\
 &\quad+\int_0^t \mathbb E[V_{P^{-1}}(x(s))] (\left\|\tilde u(s)\right\|/\delta)^2 ds\\
 &\leq \|\mathcal U^{\frac{1}{2}} u\|_{L^2_t}^2 +\int_0^t \mathbb E[V_{P^{-1}}(x(s))] (\left\|\tilde u(s)\right\|/\delta)^2 ds
 \end{align*}
for all $t\in[0, T]$ using that $U_{P^{-1}}$ is positive definite. Gronwall's inequality in Lemma \ref{gron_int} yields \begin{align*}
  \mathbb E[V_{P^{-1}}(x(t))] \leq \|\mathcal U^{\frac{1}{2}} u\|_{L^2_t}^2 \expn^{\int_0^t (\left\|\tilde u(s)\right\|/\delta)^2 ds}\leq \|\mathcal U^{\frac{1}{2}} u\|_{L^2_T}^2 \expn^{(\left\|\tilde u\right\|_{L^2_T}/\delta)^2}                                                                                                                        \end{align*}
for all $t\in[0, T]$. On the other hand, we find that
 \begin{align*}
 \langle x(t), p_{k}  \rangle^2 &\leq \lambda_{k}\; \sum_{i=1}^n \lambda_{i}^{-1} \langle x(t), p_{i}
\rangle^2 =\lambda_{k} \Big\|\sum_{i=1}^n \lambda_{i}^{-\frac{1}{2}} \langle x(t), p_{i}  \rangle \;p_{i}\Big\|^2
 =\lambda_{k} \Big\|P^{-\frac{1}{2}} x(t)\Big\|^2\\
 &= \lambda_{k} \; V_{P^{-1}}(x(t)).
\end{align*}
This concludes the proof.
\end{proof}
Theorem \ref{energy_est} is crucial for detecting direction $p_k$ that are of low relevance. We can directly see from this result that
 the coefficient $\langle x(\cdot), p_{k}  \rangle$ is small if the eigenvalue $\lambda_k$ is small. Based on \eqref{eigen_rep} this tells us that we can neglect the eigenvector $p_k$ if $\lambda_k$ is little. Thus, eigenspaces of $P$ associated to small eigenvalues can be truncated without causing a significant error. Moreover, it is interesting to notice that Theorem \ref{energy_est} relies on simplified reachability Gramians only. There seems to be no additional benefit in considering a more general full  reachability Gramian instead when making the above argument concerning dominant subspaces. However, we will see in Section \ref{sec_propo_red_sys} that the error bound in Theorem \ref{thm_error_bound} cannot be proved when a simplified Gramian $P$ is used. Instead we need a non simplified version of a reachability Gramian.

\subsubsection{Directions with low contributions to the output}\label{subsecQ}

Let $(q_k)$ be a basis of eigenvectors of the observability Gramian $Q$ introduced in Definition \ref{def_Q}, such that we can write \begin{align}\label{eigen_rep_Q}
x(t_0)= \sum_{k=1}^n \langle x(t_0), q_{k}
\rangle \,q_k.                                                                                                                                                                                                                                                                                                                                                                                                                                                                                                                                                                                                                                                                                                                                                                                                                                                                                                                                                                                                                                                        \end{align}
We ask two questions: \begin{itemize}
   \item[(i)] What is the impact of $\langle x(t_0), q_{k}
\rangle \,q_k$ to the output $y$ on an interval $[t_0, T]$ with $t_0\in [0, T)$ when $B=0$, $E=0$ and $M\equiv 0$? This means that we investigate how much output energy is produced when starting in $\langle x(t_0), q_{k}
\rangle \,q_k$ at $t_0$ if we switch off the additive controls.
                       \item[(ii)] How much is the output $y$ affected on $[t_0, T]$ if we remove $\langle x(t_0), q_{k} \rangle \,q_k$ in \eqref{eigen_rep_Q} at $t_0$?
                      \end{itemize}
We will point out that both questions are equivalent in the linear case. The first question is answered by using a simplified observability Gramian. The second, and probability more relevant question, can only be answered using a full observability Gramian.

Let us consider \eqref{original_system} starting at time $t_0\in [0, T)$, i.e.,
\begin{align}\nonumber
 \mathbf x(t)&=\mathbf x_{t_0}+\int_{t_0}^t[f(\mathbf x(s))+Bu(s)+ G(\mathbf x(s))u(s)]ds+\int_{t_0}^t[\Gamma(\mathbf x(s))+M(u(s))]dw(s),\\ \label{stochstate_mod_in}
            \mathbf y(t) &= h(\mathbf x(t))+E u(t),\quad t\in [t_0, T].
\end{align}
Clearly, we have $x(t)=\mathbf x(t)$ for all $t\geq t_0$ by the additivity of the integrals if $\mathbf x_{t_0}=x(t_0)$. We use \eqref{stochstate_mod_in} to clarify which directions in $x(t_0)$ are less relevant for $y$ on $[t_0, T]$.
Vectors $q_k$ with a low contribution to the output shall be identified. Therefore, we begin with analyzing \eqref{stochstate_mod_in}, where we set $\mathbf x_{t_0}=\langle x(t_0), q_{k}
\rangle \,q_k$ to find an answer on how significant this part of $x(t_0)$ is.
\begin{thm}\label{energy_est2}
Let $Q$ be a simplified observability Gramian according to Definition \ref{def_Q} and  $\mathbf y$ be the output variable of system  \eqref{stochstate_mod_in}. We further assume for this system that $B=0$, $E=0$ and $M\equiv 0$. Let $(q_k, \mu_{k})$ be a basis of eigenvectors of $Q$, like in \eqref{eigen_rep_Q}, equipped with the eigenvalues $\mu_{k}=\mu_k(\delta)>0$, $\delta>0$. Assuming $\mathbf x_{t_0}=\langle x(t_0), q_{k}
\rangle q_k$ and that $\tilde u$, defined in \eqref{multipl_control}, is deterministic, we have
 \begin{align*}
   \int_{t_0}^T \mathbb E\left\|\mathbf y(s)\right\|^2 ds  \leq \mu_k   \mathbb E[\langle x(t_0), q_{k}\rangle^2] \exp\left\{\int_{t_0}^T \left\|\tilde u(s)\right\|^2 ds /\delta^2\right\}.
\end{align*}
\end{thm}
\begin{proof}
We replace $P^{-1}$ by $Q$ in \eqref{eq_after_ito_exp} and \eqref{est_bil_term}. Moreover, we set $B=0$ and $M=0$ and combine both relations. Therefore, we find that  \begin{align}\nonumber
 \frac{d}{dt}\mathbb E[V_{Q}(\mathbf x(t))]&=\mathbb E\left[2\left\langle Q \mathbf x(t), f(\mathbf x(t))+G(\mathbf x(t))u(t)\right\rangle\right] +  \mathbb E\left[\trace\left(\Gamma(\mathbf x(t)) K \Gamma(\mathbf x(t))^\top Q\right) \right]  \\ \label{eq_after_ito_exp2} &\leq \mathbb E\left[2\left\langle Q \mathbf x(t), f(\mathbf x(t))\right\rangle\right] +  \mathbb E\left[\|Q^{\frac{1}{2}}(\Gamma(\mathbf x(t))) K^{\frac{1}{2}}\|_F^2\right]\\ \nonumber
 &\quad+\mathbb E\left[\delta^2\left\|Q^{\frac{1}{2}} G(\mathbf x(t))\right\|^2_F\right] + \mathbb E\left[V_{Q}(\mathbf x(t))\right] (\left\|\tilde u(t)\right\|/\delta)^2.
 \end{align}
We integrate \eqref{eq_after_ito_exp2} over $[t_0, t]$ resulting in
\begin{align*}
 \mathbb E[V_{Q}(\mathbf x(t))]-\mathbb E[V_{Q}(\mathbf x(t_0))]&\leq  \int_{t_0}^t \mathbb E\left[LV_{Q}(\mathbf x(s), 0; \delta)\right]+\mathbb E[V_{Q}(\mathbf x(s))] (\left\|\tilde u(s)\right\|/\delta)^2 ds\\
 &\leq \int_{t_0}^t -\mathbb E\left\|\mathbf y(s)\right\|^2+\mathbb E[V_{Q}(\mathbf x(s))] (\left\|\tilde u(s)\right\|/\delta)^2 ds
 \end{align*}
for all $t\in[t_0, T]$, where we inserted \eqref{inequalityQ} above and exploited that $E=0$. We apply Lemma \ref{gron_int} with $\alpha(t) = \mathbb E[V_{Q}(\mathbf x(t_0))]- \int_{t_0}^t \mathbb E\left\|\mathbf y(s)\right\|^2 ds\leq \mathbb E[V_{Q}(\mathbf x(t_0))]= \mathbb E[\langle x(t_0), q_{k}
\rangle^2 ]\mu_k$ and obtain  \begin{align*}
   0\leq \mathbb E[V_{Q}(\mathbf x(t))]\leq \alpha(t)+\int_{t_0}^t \alpha(s)(\left\|\tilde u(s)\right\|/\delta)^2 \exp\left\{\int_s^t (\left\|\tilde u(v)\right\|/\delta)^2 dv\right\} ds
    \end{align*}
 for all $t\in[t_0, T]$. Consequently, we have
    \begin{align*}
   \int_{t_0}^t \mathbb E\left\|\mathbf y(s)\right\|^2 ds  &\leq \mathbb E[\langle x(t_0), q_{k}\rangle^2] \mu_k \left(1+  \int_{t_0}^t (\left\|\tilde u(s)\right\|/\delta)^2 \exp\left\{\int_s^t (\left\|\tilde u(v)\right\|/\delta)^2 dv\right\} ds\right)\\
   &=\mathbb E[\langle x(t_0), q_{k}\rangle^2] \mu_k   \exp\left\{\int_{t_0}^t (\left\|\tilde u(s)\right\|/\delta)^2 ds\right\}
   \end{align*}
 for all $t\in[t_0, T]$. Setting $t=T$ yields the claim.
 \end{proof}
If an eigenvalue $\mu_k$ corresponding to an eigenvector $q_k$ of a simplified observability Gramian $Q$ is small, then we know that the output will be small if we initialize the system in the direction of $q_k$ while switching off the additive controls. This is a motivation to neglect $q_k$ in the dynamics. However, it is more satisfactory to know what happens to the fully controlled system if we remove $q_k$ from $x(t_0)$. Addressing this second question, we determine the equation for the difference between \eqref{original_system} and \eqref{stochstate_mod_in}. We now set \begin{align}\label{initial_state_diff}
  \mathbf x_{t_0}=\sum_{i=1 \atop i\neq k}^n \langle x(t_0), q_{i}
\rangle \,q_i,                                                                                                                                                                                                                                  \end{align}
i.e, we investigate how much the output $y$ is affected if we neglect the direction $q_k$ in $x(t_0)$. We find that
\begin{equation}\label{stochstate_mod_in2}
\begin{aligned}
 x(t)-\mathbf x(t)&=x(t_0)-\mathbf x_{t_0}+\int_{t_0}^t[f(x(s))-f(\mathbf x(s))+ [G(x(s))-G(\mathbf x(s))]u(s)]ds\\
 &\quad\quad\quad\quad\quad\quad+\int_{t_0}^t[\Gamma(x(s))-\Gamma(\mathbf x(s))]dw(s),\\
          y(t)-\mathbf y(t)&=h(x(t))- h(\mathbf x(t)),\quad t\in [t_0, T].
\end{aligned}
\end{equation}
We use \eqref{stochstate_mod_in2} in the following theorem.
\begin{thm}\label{energy_est3}
Let $Q$ be an observability Gramian according to Definition \ref{def_Q}, $y$ be the output of \eqref{original_system} and  $\mathbf y$ be that of system
 \eqref{stochstate_mod_in} with initial state $\mathbf x_{t_0}$ given in \eqref{initial_state_diff}. Let $(q_k, \mu_{k})$ be a basis of eigenvectors of $Q$, like in \eqref{eigen_rep_Q}, equipped with the eigenvalues $\mu_{k}=\mu_k(\delta)>0$, $\delta>0$. Assuming  that $\tilde u$, defined in \eqref{multipl_control}, is deterministic, we have
 \begin{align*}
   \int_{t_0}^T \mathbb E\left\|y(s)-\mathbf y(s)\right\|^2 ds  \leq \mu_k   \mathbb E[\langle x(t_0), q_{k}\rangle^2] \exp\left\{\int_{t_0}^T \left\|\tilde u(s)\right\|^2 ds /\delta^2\right\}.
\end{align*}
\end{thm}
\begin{proof}
Considering \eqref{stochstate_mod_in2}, we apply Lemma \ref{lemstochdiff} for  $V(x)=V_Q(x)=x^\top Q x$ setting $a(t)=f(x(t))-f(\mathbf x(t))+[G(x(t))-G(\mathbf x(t))]u(t)$ and $B(t)= \Gamma(x(t))-\Gamma(\mathbf x(t))$. Taking the expectation, we  have \begin{align}\nonumber
 \frac{d}{dt}\mathbb E[V_{Q}(x(t)-\mathbf x(t))]=&\mathbb E\left[2\left\langle Q (x(t)-\mathbf x(t)), f(x(t))-f(\mathbf x(t))+[G(x(t))-G(\mathbf x(t))]u(t)\right\rangle\right]\\ \label{first_est_withQ}
 &+  \mathbb E\left[\trace\left((\Gamma(x(t))-\Gamma(\mathbf x(t))) K (\Gamma(x(t))-\Gamma(\mathbf x(t)))^\top Q\right) \right].                                                                                                                   \end{align}
 We make use of the same steps as in estimate \eqref{estimate_on_G} yielding \begin{align*}
   &2\left\langle Q (x(t)-\mathbf x(t)), [G(x(t))-G(\mathbf x(t))]u(t)\right\rangle\\
   &\leq
  \delta^2\left\|Q^{\frac{1}{2}}\big(G(x(t))-G(\mathbf x(t))\big)\right\|^2_F + V_{Q}(x(t)-\mathbf x(t)) (\left\|\tilde u(t)\right\|/\delta)^2.
                 \end{align*}
Combining this inequality with \eqref{first_est_withQ} and integrating the result over $[t_0, t]$ gives us \begin{align*}
 &\mathbb E[V_{Q}(x(t)-\mathbf x(t))]-\mathbb E[V_{Q}(x(t_0)-\mathbf x(t_0))]=\mathbb E[V_{Q}(x(t)-\mathbf x(t))]- \mathbb E[\langle x(t_0), q_{k} \rangle^2 ]\mu_k\\
 &\leq  \int_{t_0}^t \mathbb E\left[LV_{Q}(x(t), \mathbf x(s); \delta)\right]+\mathbb E[V_{Q}(x(s)-\mathbf x(s))] (\left\|\tilde u(s)\right\|/\delta)^2 ds\\
 &\leq \int_{t_0}^t -\mathbb E\left\|y(s)-\mathbf y(s)\right\|^2+\mathbb E[V_{Q}(x(t)-\mathbf x(s))] (\left\|\tilde u(s)\right\|/\delta)^2 ds
 \end{align*}
for all $t\in[t_0, T]$. The last inequality is achieved by inserting \eqref{inequalityQ} above and by the output equation in \eqref{stochstate_mod_in2}. We apply Lemma \ref{gron_int} with $\alpha(t) =  \mathbb E[\langle x(t_0), q_{k} \rangle^2 ]\mu_k- \int_{t_0}^t \mathbb E\left\|y(s)-\mathbf y(s)\right\|^2 ds$. Following the exactly same arguments as in the proof of Theorem \ref{energy_est2}, we find that  \begin{align*}
   \int_{t_0}^t \mathbb E\left\|y(s)-\mathbf y(s)\right\|^2 ds  \leq \mathbb E[\langle x(t_0), q_{k}\rangle^2] \mu_k   \exp\left\{\int_{t_0}^t (\left\|\tilde u(s)\right\|/\delta)^2 ds\right\}
   \end{align*}
  for all $t\in [t_0, T]$ yielding the claim.
  \end{proof}
Theorem \ref{energy_est3} tells us that the output $\mathbf y$ that is obtained after neglecting the eigenvector $q_k$ of the full observability Gramian is still close to the original output $y$ if the associated eigenvalue $\mu_k$ is small. This is a strong indicator that we do not need such a $q_k$ in our dynamics. Given that \eqref{original_system} has linear coefficients $f, G, \Gamma$ and $h$, system \eqref{stochstate_mod_in2} with $\mathbf x_{t_0}$ like in \eqref{initial_state_diff} becomes
\begin{align*}
 x(t)-\mathbf x(t)&=\langle x(t_0), q_{k} \rangle \,q_k+\int_{t_0}^t[f\big(x(s)-\mathbf x(s)\big)+ G\big(x(s)-\mathbf x(s)\big)u(s)]ds\\
 &\quad\quad\quad\quad\quad\quad\quad\;+\int_{t_0}^t \Gamma\big(x(s)-\mathbf x(s)\big) dw(s),\\
          y(t)-\mathbf y(t)&=h\big(x(t)- \mathbf x(t)\big),\quad t\in [t_0, T].
\end{align*}
Therefore, $x-\mathbf x$ is the solution of the state equation in \eqref{stochstate_mod_in} with initial value $\langle x(t_0), q_{k} \rangle$ given $B=0$ and $M\equiv 0$. Moreover, given $E=0$, $y(t)-\mathbf y(t)$ is the corresponding output that is addressed in Theorem \ref{energy_est2}. For that reason, Theorems \ref{energy_est2} and \ref{energy_est3} consider the same aspect. Hence, both questions asked in this subsection are the same in the linear case. This might not be surprising given the fact that the simplified and the full Gramian are the same for linear system, see Remark \ref{remark_lin_case}.

\section{Reduced order model based on balancing and truncation}\label{sec_BT}

We found unnecessary system information in Subsections \ref{subsecP} and \ref{subsecQ}. Therefore, we know that eigenspaces of $P$ and $Q$ can be truncated if the associated eigenvalues are small. Since the eigenspaces of $P$ and $Q$ do not coincide, we apply a state space transform that ensures that $P$ and $Q$ are both diagonal and the same which allows a simple and simultaneous truncation of unimportant information. For that reason, we introduce an invertible matrix $\mathcal S\in\mathbb{R}^{n\times n}$ that defines a new state by $\tilde x=\mathcal Sx$. Inserting this into \eqref{original_system}  leads to an equivalent stochastic system  \begin{equation}\label{original_system_S}
\begin{aligned}
             d\tilde x(t)&=[\tilde f(\tilde x(t))+\tilde Bu(t)+ \tilde G(\tilde x(t))u(t)]dt+[\tilde \Gamma(\tilde x(t))+\tilde M(u(t))]dw(t),\\ 
            y(t) &= \tilde h(\tilde x(t))+E u(t),\quad t\geq 0,
\end{aligned}
\end{equation}
with coefficients
\begin{align}\label{coef_trans}
  \Big(\tilde f, \tilde B, \tilde G, \tilde \Gamma, \tilde M, \tilde h\Big):=\Big(\mathcal S f(\mathcal S^{-1}\cdot), \mathcal S B, \mathcal S G(\mathcal S^{-1}\cdot), \mathcal S \Gamma(\mathcal S^{-1}\cdot), \mathcal S M(\cdot), h(\mathcal S^{-1}\cdot)\Big).
\end{align}
System \eqref{original_system_S} has the same input-output behaviour like \eqref{original_system}. On the other hand, the Gramians change under this transformation. We can use this fact to simultaneously diagonalize $P$ and $Q$. This diagonalization is called balancing.
\begin{prop}\label{prop_bal}
 Let $\mathcal S$ be an invertible matrix and $P$, $Q$ Gramians of \eqref{original_system} according to Definitions \ref{def_P} and \ref{def_Q}, respectively. Then, $\tilde P= \mathcal SP\mathcal S^\top$ and $\tilde Q=\mathcal S^{-\top}Q\mathcal S^{-1}$ are Gramians of \eqref{original_system_S}. Given that $Q$ is invertible, we obtain   $\tilde P=\tilde Q = \Sigma= \diag(\sigma_1,\ldots,\sigma_n)$ using  \begin{align}\label{bal_transform_new}
       \mathcal S=\Sigma^{\frac{1}{2}} V^\top L_P^{-1},                                                                                                                                                                                                                                                                                                                                                                                                                                                                                                                                                                                                                                                                                                                                                                                          \end{align}
where $P=L_PL_P^\top$ and $L_P^\top QL_P=V\Sigma^2 V^\top$ is an eigenvalue decomposition with an orthogonal $V$. $\mathcal S$ in \eqref{bal_transform_new} is called balancing transformation.
\end{prop}
\begin{proof}
Using Proposition \ref{prop_lyap}, we have \begin{align*}
LV_{X}(x, y; \delta)
&=2\langle X(x-y), f(x)-f(y)\rangle + \delta^2\trace\left((G(x)-G(y)) (G(x)-G(y))^\top X\right)\\ &\quad+ \trace\left((\Gamma(x)-\Gamma(y)) K (\Gamma(x)-\Gamma(y))^\top X\right)
\end{align*}
for $x, y\in\mathbb R^n$. We observe that \begin{align*}
 \langle X(x-y), f(x)-f(y)\rangle &=\langle \mathcal S^{-\top}X\mathcal S^{-1} \mathcal S(x-y), \mathcal S(f(\mathcal S^{-1} \mathcal Sx)-f(\mathcal S^{-1} \mathcal Sy))\rangle\\
 &= \langle \mathcal S^{-\top}X\mathcal S^{-1}(\tilde x-\tilde y), \tilde f(\tilde x)-\tilde f(\tilde y)\rangle
                                              \end{align*}
with $\tilde x=\mathcal Sx$ and $\tilde y= \mathcal Sy$. We obtain that \begin{align*}
&\trace\left((\Gamma(x)-\Gamma(y)) K (\Gamma(x)-\Gamma(y))^\top X\right)\\
&=\trace\left(\mathcal S(\Gamma(\mathcal S^{-1} \mathcal Sx)-\Gamma(\mathcal S^{-1} \mathcal Sy)) K (\Gamma(\mathcal S^{-1} \mathcal Sx)-\Gamma(\mathcal S^{-1} \mathcal Sy))^\top \mathcal S^\top \mathcal S^{-\top}X\mathcal S^{-1}\right)\\
&=\trace\left((\tilde\Gamma(\tilde x)-\tilde \Gamma(\tilde y)) K (\tilde \Gamma(\tilde x)-\tilde \Gamma(\tilde y))^\top  \mathcal S^{-\top}X\mathcal S^{-1}\right).
                                    \end{align*}
The same relation for the trace expression involving $G$ follows completely analogously. Now, let $\tilde L$ denote the operator that is obtained by replacing $( f, G, \Gamma)$ by $(\tilde f, \tilde G, \tilde \Gamma)$ in \eqref{layp_op}. Hence, we have \begin{align}\label{trans1}
 LV_{X}(x, y; \delta)= \tilde LV_{\mathcal S^{-\top}X\mathcal S^{-1}}(\tilde x, \tilde y; \delta).                                                                                                                                                                                                                                                                                                                                                                                                                                                                                                                                                                                                       \end{align}
Recalling the definitions of $U$ and $S$ in \eqref{def_U} and \eqref{def_S}, we find\begin{align}\label{trans2}
 U_{X} = \mathcal U - \sum_{i, j=1}^q M_i^\top \mathcal S^\top \mathcal S^{-\top}X\mathcal S^{-1} \mathcal S M_j k_{ij}=: \tilde U_{\mathcal S^{-\top}X\mathcal S^{-1}}
                                        \end{align}
as well as
\begin{align}\label{trans3}
    S_{X}(x,y)&=B^\top \mathcal S^{\top}\mathcal S^{-\top}X\mathcal S^{-1}\mathcal S(x-y)\\ \nonumber
    &\quad+\sum_{i, j=1}^q M_i^\top \mathcal S^{\top} \mathcal S^{-\top}X\mathcal S^{-1} \mathcal S (\gamma_j( \mathcal S^{-1} \mathcal S x)-\gamma_j(\mathcal S^{-1} \mathcal S y))k_{ij}
    =: \tilde S_{\mathcal S^{-\top}X\mathcal S^{-1}}(\tilde x, \tilde y).
\end{align}
Given that \eqref{inequalityP} holds and setting $X=P^{-1}$ in \eqref{trans1}, \eqref{trans2} and \eqref{trans3}, we have \begin{align*}
 \tilde LV_{\tilde P^{-1}}(\tilde x, \tilde y; \delta)\leq -\tilde S_{\tilde P^{-1}}(\tilde x, \tilde y)^\top \tilde U_{\tilde P^{-1}}^{-1} \tilde S_{\tilde P^{-1}}(\tilde x, \tilde y)
                      \end{align*}
for all $\tilde x, \tilde y\in\mathbb R^n$ and some $\delta>0$. This means that $\tilde P=\mathcal SP\mathcal S^\top$
 is a reachability Gramian of \eqref{original_system_S}. Now, fixing $X=Q$ in \eqref{trans1}, we obtain from \eqref{inequalityQ} that \begin{align*}
 \tilde LV_{\tilde Q}(\tilde x, \tilde y; \delta)\leq -(\tilde h(\tilde x)- \tilde h(\tilde y))^\top (\tilde h(\tilde x)- \tilde h(\tilde y))
                      \end{align*}
for all $\tilde x, \tilde y\in\mathbb R^n$ and some $\delta>0$, i.e, $\tilde Q=\mathcal S^{-\top}Q\mathcal S^{-1}$ is an observability Gramian of \eqref{original_system_S}.
We insert \eqref{bal_transform_new} into $\tilde P= \mathcal SP\mathcal S^\top= \Sigma^{\frac{1}{2}} V^\top L_P^{-1} P L_P^{-\top} V \Sigma^{\frac{1}{2}}=\Sigma$ and use the same $\mathcal S$ for $\tilde Q=\mathcal S^{-\top}Q\mathcal S^{-1}= \Sigma^{-\frac{1}{2}} V^\top L_P^{\top} Q L_P V \Sigma^{-\frac{1}{2}}=\Sigma$ which concludes the proof.
\end{proof}
The diagonal entries of the balanced system  Gramians $\Sigma$ are $\sigma_i= \sqrt{\lambda_i(PQ)}$. We call $\sigma_1, \dots, \sigma_n$ Hankel singular values (HSVs) in the following.
 We partition $\Sigma=\diag(\Sigma_1,\Sigma_{2})$, where
  $\Sigma_1= \diag(\sigma_1,\ldots,\sigma_r)$ contains large and $\Sigma_{2}=\diag(\sigma_{r+1},\ldots,\sigma_n)$, $r\ll n$, small HSVs. The state \begin{align*}
  \tilde x=\begin{bmatrix}
 x_{1}\\  x_{2}
\end{bmatrix}
\end{align*}
in \eqref{original_system_S} is divided into the important variables $x_1$ (taking values in $\mathbb R^r$) associated to $\Sigma_1$ and into the variables $x_2$ of low relevance corresponding to $\Sigma_2$. Before stating the reduced order system, we introduce truncated versions $f_r, G_r, \Gamma_r$ of the vector/matrix-valued functions $\tilde f, \tilde G, \tilde \Gamma$ on $\mathbb R^n$ by \begin{align}\label{def_trucation_function}
  \tilde f(\smat x_r\\ 0\srix) = \begin{bmatrix}{f}_r(x_r)\\ f_\star(x_r)\end{bmatrix},\quad  \tilde G(\smat x_r\\ 0\srix) = \begin{bmatrix}{G}_r(x_r)\\ G_\star(x_r)\end{bmatrix},\quad  \tilde {\Gamma}(\smat x_r\\ 0\srix) = \begin{bmatrix}{\Gamma}_r(x_r)\\ \Gamma_\star(x_r)\end{bmatrix}                                                                                                                                                                                                                                                                                                                                                                                                                                                                                                                                                                                                                    \end{align}
with $x_r\in \mathbb R^r$, and $0\in \mathbb R^{n-r}$. Here, $f_r, G_r, \Gamma_r$ are vector/matrix-valued maps on $\mathbb R^r$ having $r$ rows.
Moreover, we define the remaining reduced order coefficients as follows
  \begin{equation}\label{part_bal_control_part}
  \begin{aligned}
 h_r(x_r)=\tilde h(\smat x_r\\ 0\srix),\quad\tilde B &= \begin{bmatrix}{B}_r\\ B_\star\end{bmatrix}\quad \text{and}\quad\tilde M(z) = \begin{bmatrix}{M}_r(z)\\ M_\star(z)\end{bmatrix},\quad z\in\mathbb R^m.
  \end{aligned}
  \end{equation}
Again, $B_r$ and $M_r$ have $r$ rows.

The reduced system is obtained by setting $x_2\equiv 0$ in \eqref{original_system_S} and by removing the equations of $x_{2}$. This results in a reduced system
\begin{subequations}\label{red_system}
 \begin{align}\label{red_system_state}
             dx_r(t)&=[f_r(x_r(t))+B_ru(t)+ G_r(x_r(t))u(t)]dt+[\Gamma_r(x_r(t))+M_r(u(t))]dw(t),\\ \label{red_output}
            y_r(t) &= h_r(x_r(t))+E u(t),\quad t\geq 0.
\end{align}
\end{subequations}

\section{Properties of the reduced system \eqref{red_system}}\label{sec_propo_red_sys}

For simplicity of the notation, let us from now on assume that the original system is already balanced, i.e., the balancing transformation using \eqref{bal_transform_new} has already been applied, so that \eqref{original_system} and \eqref{original_system_S} coincide. Consequently, we have $P=Q=\Sigma=\diag(\sigma_1, \dots, \sigma_n)$, so that \eqref{inequalityP} and \eqref{inequalityQ} become
\begin{subequations}\label{balanced_equations}
\begin{align}\label{diagP}
LV_{\Sigma^{-1}}(x, y; \delta)&\leq -S_{\Sigma^{-1}}(x,y)^\top U_{\Sigma^{-1}}^{-1} S_{\Sigma^{-1}}(x, y)\\ \label{diagQ}
 LV_{\Sigma}(x, y; \delta)&\leq -(h(x)- h(y))^\top (h(x)- h(y)),                                                                                                                                                     \end{align}
 \end{subequations}
for all $x, y\in\mathbb R^n$ and some $\delta>0$, while $U_{\Sigma^{-1}}$ is positive definite.

We now derive the Gramians for \eqref{red_system} based on \eqref{balanced_equations}.
\begin{prop}\label{prop_red_bal}
The diagonal matrix $\Sigma_1=\diag(\sigma_1, \dots, \sigma_r)$ is both  a reachability and an observability Gramian of \eqref{red_system} according to Definitions \ref{def_P} and \ref{def_Q}.
\end{prop}
\begin{proof}
We assume that \eqref{original_system} is balanced meaning that the balancing transformation in \eqref{bal_transform_new} has already been applied while keeping the notation of the orginal system. Consequently, we have $\Big(f, B, G, \Gamma, M, h\Big)=\Big(\tilde f, \tilde B, \tilde G, \tilde \Gamma, \tilde M, \tilde h\Big)$ in \eqref{original_system_S}. Inequality \eqref{diagP} implies that \begin{align}\nonumber
     LV_{\Sigma^{-1}}(x, y; \delta)&\leq -S_{\Sigma^{-1}}(x,y)^\top U_{\Sigma^{-1}}^{-1} S_{\Sigma^{-1}}(x, y)+\|U_{\Sigma^{-1}}^{\frac{1}{2}}z+U_{\Sigma^{-1}}^{-\frac{1}{2}}S_{\Sigma^{-1}}(x,y) \|^2\\ \label{id1}
     &= z^\top U_{\Sigma^{-1}} z +2 S_{\Sigma^{-1}}(x,y)^\top z
                      \end{align}
for all $x, y\in\mathbb R^n$, all $z\in \mathbb R^m$ and some $\delta>0$. We choose $x=\smat x_1\\
 0\srix$, $y=\smat y_1\\
 0\srix$ for $x_1, y_1\in\mathbb R^r$ and partition \begin{align*}
f(\smat x_1\\ 0\srix)-f(\smat y_1\\ 0\srix)  &= \begin{bmatrix}{f}_r(x_1)-{f}_r(y_1)\\\star_f\end{bmatrix}, \quad \Gamma(\smat x_1\\ 0\srix)-\Gamma(\smat y_1\\ 0\srix) = \begin{bmatrix}{\Gamma}_r(x_1)-{\Gamma}_r(y_1)\\\star_\Gamma\end{bmatrix}, \\
G(\smat x_1\\ 0\srix)-G(\smat y_1\\ 0\srix) &= \begin{bmatrix}{G}_r(x_1)-{G}_r(y_1)\\\star_G\end{bmatrix}
                               \end{align*}
using \eqref{def_trucation_function} while setting $\star_f:= f_\star(x_1)-f_\star(y_1)$, $\star_G:= G_\star(x_1)-G_\star(y_1)$ and $\star_\Gamma:= \Gamma_\star(x_1)-\Gamma_\star(y_1)$. With the partition of $\Sigma$ and the representation in Proposition \ref{prop_lyap}, we obtain \begin{align}\nonumber
  LV_{\Sigma^{-1}}(\smat x_1\\ 0\srix, \smat y_1\\ 0\srix; \delta)  & =  2\langle \Sigma_1^{-1}(x_1-y_1), f_r(x_1)-f_r(y_1)\rangle \\ \nonumber &+ \delta^2\trace\left((G_r(x_1)-G_r(y_1))^\top \Sigma_1^{-1} (G_r(x_1)-G_r(y_1)) \right)+ \delta^2\trace\left(\star_G^\top \Sigma_2^{-1} \star_G \right)\\ \nonumber &+ \trace\left(K (\Gamma_r(x_1)-\Gamma_r(y_1))^\top \Sigma_1^{-1} (\Gamma_r(x_1)-\Gamma_r(y_1))\right) + \trace\left(K \star_\Gamma^\top \Sigma_2^{-1} \star_\Gamma\right)   \\ \label{id2}
  &=  L_r V_{\Sigma_1^{-1}}(x_1, y_1; \delta)  + \delta^2\trace\left(\star_G^\top \Sigma_2^{-1} \star_G \right)+ \trace\left(K \star_\Gamma^\top \Sigma_2^{-1} \star_\Gamma\right)
  \end{align}
for all $x_1, y_1\in\mathbb R^r$ and some $\delta > 0$, where $L_r$ is the operator obtained when replacing $(f, G, \Gamma)$ by $(f_r, G_r, \Gamma_r)$ in \eqref{layp_op}. Below, we exploit \eqref{part_bal_control_part} and use \eqref{traceU_rel} with $X=\Sigma^{-1}$. Therefore, we find that
     \begin{align}\label{id3}
z^\top U_{\Sigma^{-1}} z &= z^\top \mathcal U z-\trace\left(K M(z)^\top \Sigma^{-1} M(z)\right) \\ \nonumber
&= z^\top \mathcal U z-\trace\left(K M_r(z)^\top \Sigma_1^{-1} M_r(z)\right)-\trace\left(K M_\star(z)^\top \Sigma_2^{-1} M_\star(z)\right).
\end{align}
Based on \eqref{traceS_rel} with $X=\Sigma^{-1}$ we have \begin{align}\nonumber
 z^\top  S_{\Sigma^{-1}}(\smat x_1\\ 0\srix, \smat y_1\\ 0\srix) = z^\top B_r^\top \Sigma_1^{-1} (x_1-y_1) &+ \trace\left(K M_r(z)^\top \Sigma_1^{-1}(\Gamma_r(x_1) - \Gamma_r(y_1))\right)\\ \label{id4}
 &+ \trace\left(K M_\star(z)^\top \Sigma_2^{-1}\star_\Gamma\right).
\end{align}
Inserting \eqref{id2}, \eqref{id3} and \eqref{id4} into \eqref{id1} yields
{\allowdisplaybreaks
\begin{align*}
 & L_r V_{\Sigma_1^{-1}}(x_1, y_1; \delta)  + \delta^2\trace\left(\star_G^\top \Sigma_2^{-1} \star_G \right)+ \trace\left(K \star_\Gamma^\top \Sigma_2^{-1} \star_\Gamma\right)\\
 &\leq z^\top \mathcal U z-\trace\left(K M_r(z)^\top \Sigma_1^{-1} M_r(z)\right)-\trace\left(K M_\star(z)^\top \Sigma_2^{-1} M_\star(z)\right)+2 z^\top B_r^\top \Sigma_1^{-1} (x_1-y_1) \\
 &\quad + 2\trace\left(K M_r(z)^\top \Sigma_1^{-1}(\Gamma_r(x_1) - \Gamma_r(y_1))\right)
 + 2\trace\left(K M_\star(z)^\top \Sigma_2^{-1}\star_\Gamma\right).                                                                                                                                                                                                   \end{align*}}
 Since we can observe that $\trace\left(K \star_\Gamma^\top \Sigma_2^{-1} \star_\Gamma\right)+ \trace\left(K M_\star(z)^\top \Sigma_2^{-1} M_\star(z)\right)- 2 \trace\left(K M_\star(z)^\top \Sigma_2^{-1}\star_\Gamma\right) =\|\Sigma_2^{-\frac{1}{2}}(\star_\Gamma- M_\star(z))K^{\frac{1}{2}}\|_F^2\geq 0$, we obtain
 \begin{align*}
  L_r V_{\Sigma_1^{-1}}(x_1, y_1; \delta)
 &\leq  z^\top \mathcal U z-\trace\left(K M_r(z)^\top \Sigma_1^{-1} M_r(z)\right)\\
 &\quad +2 z^\top B_r^\top \Sigma_1^{-1} (x_1-y_1) + 2\trace\left(K M_r(z)^\top \Sigma_1^{-1}(\Gamma_r(x_1) - \Gamma_r(y_1))\right)                                                                                                                                                                                                   \end{align*}
 for all $x_1, y_1\in\mathbb R^r$, all $z\in \mathbb R^r$ and some $\delta>0$. Let $U_{r, \cdot}$ and $S_{r, \cdot}$ be the two objects that we obtain by replacing the coefficients of \eqref{original_system} by the coefficients of \eqref{red_system} in \eqref{def_U} and \eqref{def_S}. Exploiting \eqref{traceU_rel} and \eqref{traceS_rel} for the reduced system and setting $z=- U_{r, \Sigma_1^{-1}}^{-1} S_{r, \Sigma_1^{-1}}(x_1, y_1)$ leads to \begin{align*}
  L_r V_{\Sigma_1^{-1}}(x_1, y_1; \delta)
 &\leq  z^\top U_{r, \Sigma_1^{-1}} z+2S_{r, \Sigma_1^{-1}}(x_1, y_1)^\top z=- S_{r, \Sigma_1^{-1}}(x_1, y_1)^\top U_{r, \Sigma_1^{-1}}^{-1} S_{r, \Sigma_1^{-1}} (x_1, y_1)                                                                                                                                                                                                   \end{align*}
 for  all $x_1, y_1\in\mathbb R^r$ and some  $\delta>0$. This shows that $\Sigma_1$ is a reachability Gramian of the reduced system \eqref{red_system}. We set $x=\smat x_1\\
 0\srix$, $y=\smat y_1\\
 0\srix$ for $x_1, y_1\in\mathbb R^r$ in \eqref{diagQ}.  Further, use \eqref{id2} and replace $\Sigma^{-1}$ by $\Sigma$ yielding \begin{align}\nonumber
  LV_{\Sigma}(\smat x_1\\ 0\srix, \smat y_1\\ 0\srix; \delta)  =  L_r V_{\Sigma_1}(x_1, y_1; \delta)  + \delta^2\trace\left(\star_G^\top \Sigma_2 \star_G \right)+ \trace\left(K \star_\Gamma^\top \Sigma_2 \star_\Gamma\right)\geq L_r V_{\Sigma_1}(x_1, y_1; \delta)
  \end{align}
for all $x_1, y_1\in\mathbb R^r$ and some $\delta > 0$. By \eqref{part_bal_control_part}, we have  \begin{align*}
h(\smat x_1\\ 0\srix)-h(\smat y_1\\ 0\srix)  = {h}_r(x_1)-{h}_r(y_1)
\end{align*}
for all $x_1, y_1\in\mathbb R^r$. Therefore, we obtain that $\Sigma_1$ is also an observability Gramian of \eqref{red_system}. This concludes the proof.
\end{proof}
We can now find a bound for the error between the outputs of \eqref{original_system} and \eqref{red_system} under certain assumptions on the original system. Here, it is crucial that the balancing procedure is conducted based on (full) Gramians $P$ and $Q$. Their simplified versions are not sufficient for obtaining the result below.
\begin{thm}\label{thm_error_bound}
Let $y$ be the output of \eqref{original_system} with $x(0)=0$. We assume that Gramians $P$ and $Q$ introduced in Definitions \ref{def_P} and \ref{def_Q} exist, i.e., the assumptions of Theorem \ref{thm_gram_exist} hold. Moreover, suppose that  the coefficients $f$, $G$ and $\Gamma$ of the state equation \eqref{stochstatenew} are point symmetric meaning that $f(-x)=-f(x)$ for all $x\in\mathbb R^n$ etc. We further assume that \eqref{original_system} has a balanced realization meaning that, e.g., the diagonalizing state space transformation $\mathcal S$ in Proposition \ref{prop_bal} exists. Given $u\in L^2_T$, deterministic  multiplicative controls $\tilde u$ defined in \eqref{multipl_control} and the $r$-dimensional reduced system \eqref{red_system} with output $y_r$  and $x_r(0)=0$. Then, we have \begin{align*}
\left\| y-y_{r}\right\|_{L^2_T}\leq 2\sum_{k=r+1}^n\sigma_k \|\mathcal U^{\frac{1}{2}} u\|_{L^2_T} \exp\left\{0.5\left\|\tilde u\right\|_{L^2_T}^2/\delta^2\right\},                                                                                                                                                                                                                                                                                                                                                                                                                                                                                                                                                                                    \end{align*}
where $\mathcal U$ is a positive definite matrix entering \eqref{def_U} and hence the definition of $P$ in \eqref{inequalityP}. Moreover, $\delta>0$ is the parameter entering the operator $L$ in \eqref{layp_op} and consequently the inequalities for $P$ and $Q$ in \eqref{inequalityP} and \eqref{inequalityQ}.
\end{thm}
\begin{proof}
We assume that \eqref{original_system} is already balanced, i.e., $\Big(f, B, G, \Gamma, M, h\Big)=\Big(\tilde f, \tilde B, \tilde G, \tilde \Gamma, \tilde M, \tilde h\Big)$ and \eqref{balanced_equations} hold. Practically, this means that we have already applied the balancing procedure explained in Section \ref{sec_BT} (which exists by assumption), but we keep the notation of original model. Properties like the point symmetry of the coefficients are preserved under a state space transformation. Let us now recall the definitions of the reduced coefficients from \eqref{def_trucation_function} and \eqref{part_bal_control_part}:
\begin{align*}
f(\smat x_r\\ 0\srix) = \begin{bmatrix}{f}_r(x_r)\\f_\star(x_r)\end{bmatrix}, \quad \Gamma(\smat x_r\\ 0\srix) = \begin{bmatrix}{\Gamma}_r(x_r)\\ \Gamma_\star(x_r)\end{bmatrix}, \quad G(\smat x_r\\ 0\srix) = \begin{bmatrix}{G}_r(x_r)\\G_\star(x_r)\end{bmatrix}, \quad h(\smat x_r\\ 0\srix) = {h}_r(x_r)
                               \end{align*}
for $x_r\in\mathbb R^r$ and  \begin{align*}
B = \begin{bmatrix}{B}_r\\ B_\star\end{bmatrix}\quad \text{and}\quad M(z) &= \begin{bmatrix}{M}_{r}(z)\\{M}_\star(z)\end{bmatrix}
                      \end{align*}
for $z\in\mathbb R^m$.
 We add a redundant line to the reduced system \eqref{red_system} yielding
\begin{equation}\label{bal_par_minus1}
 \begin{aligned}
             d\smat x_r(t)\\0\srix=&\left[f\left(\smat x_r(t)\\0\srix\right)+Bu(t)+ G\left(\smat x_r(t)\\0\srix\right)u(t)-\smat 0 \\ v_0(t)\srix\right]dt\\
             &+\left[\Gamma\left(\smat x_r(t)\\0\srix\right)+M(u(t))-\smat 0 \\ v_1(t)\srix\right]dw(t),
\end{aligned}
\end{equation}
where $v_0= f_\star(x_r)+B_\star u+G_\star(x_r)u$ and $v_1=\Gamma_\star(x_r) + M_\star(u)$. We introduce $x_-(t):= x(t) -\smat x_{r}(t)\\0\srix $ and $x_+(t):= x(t) +\smat x_{r}(t)\\0\srix $. Combining \eqref{original_system} with \eqref{bal_par_minus1} gives us
{\allowdisplaybreaks\begin{align} \nonumber
 dx_-(t) &= \left[f(x(t))-f\left(\smat x_r(t)\\0\srix\right)+\left[G(x(t))-G\left(\smat x_r(t)\\0\srix\right)\right]u(t)+\smat 0 \\ v_0(t)\srix\right]dt
 \\ \label{eq_x_minus}&\quad + \left[\Gamma(x(t))-\Gamma\left(\smat x_r(t)\\0\srix\right)+\smat 0 \\ v_1(t)\srix\right]dw(t),\\ \nonumber
  dx_+(t) &= \left[f(x(t))+f\left(\smat x_r(t)\\0\srix\right)+2B u(t)+\left[G(x(t))+G\left(\smat x_r(t)\\0\srix\right)\right]u(t)-\smat 0 \\ v_0(t)\srix\right]dt
 \\ \label{eq_x_plus} &\quad + \left[\Gamma(x(t))+\Gamma\left(\smat x_r(t)\\0\srix\right)+2M(u(t))-\smat 0 \\ v_1(t)\srix\right]dw(t).
\end{align}}
Now, we analyze quadratic forms of $x_-$ and $x_+$, exploiting Ito's formula, to prove the desired bound. Using \eqref{eq_x_minus}, we apply Lemma \ref{lemstochdiff} for  $V(x)=V_\Sigma(x)=x^\top \Sigma x$ setting $a(t)=f(x(t))-f\left(\smat x_r(t)\\0\srix\right)+\left[G(x(t))-G\left(\smat x_r(t)\\0\srix\right)\right]u(t)+\smat 0 \\ v_0(t)\srix$ and $B(t)= \Gamma(x(t))-\Gamma\left(\smat x_r(t)\\0\srix\right)+\smat 0 \\ v_1(t)\srix$. Taking the expectation, we  find that {\allowdisplaybreaks\begin{align}\nonumber
 &\frac{d}{dt}\mathbb E[V_{\Sigma}(x_-(t))]\\ \label{first_est_withSigma}
 &=\mathbb E\left[2\left\langle \Sigma x_-(t), f(x(t))-f\left(\smat x_r(t)\\0\srix\right)+\left[G(x(t))-G\left(\smat x_r(t)\\0\srix\right)\right]u(t)+\smat 0 \\ v_0(t)\srix\right\rangle\right]\\ \nonumber
 &\quad+  \mathbb E\left[\trace\left((\Gamma(x(t))-\Gamma\left(\smat x_r(t)\\0\srix\right)+\smat 0 \\ v_1(t)\srix) K (\Gamma(x(t))-\Gamma\left(\smat x_r(t)\\0\srix\right)+\smat 0 \\ v_1(t)\srix)^\top \Sigma\right) \right].                                                                                                                   \end{align}}
Analogue to \eqref{estimate_on_G} we obtain {\allowdisplaybreaks\begin{align*}
   &2\left\langle \Sigma x_-(t), \left[G(x(t))-G\left(\smat x_r(t)\\0\srix\right)\right]u(t)\right\rangle\\
   &\leq
  \delta^2\left\|\Sigma^{\frac{1}{2}}\left[G(x(t))-G\left(\smat x_r(t)\\0\srix\right)\right]\right\|^2_F + V_{\Sigma}(x_-(t)) (\left\|\tilde u(t)\right\|/\delta)^2.
                 \end{align*}}
We insert this inequality into \eqref{first_est_withSigma}. We integrate the result over $[0, t]$. This yields {\allowdisplaybreaks\begin{align}\nonumber
 \mathbb E[V_{\Sigma}(x_-(t))]&\leq  \int_{0}^t \mathbb E\left[LV_{\Sigma}(x(s), \smat x_r(s)\\0\srix; \delta)\right]+\mathbb E[V_{\Sigma}(x_-(s))] (\left\|\tilde u(s)\right\|/\delta)^2 ds\\ \label{null}
 &\quad +\int_{0}^t \mathbb E\left[2\left\langle \Sigma x_-(s), \smat 0 \\ v_0(s)\srix\right\rangle\right]  ds
 \\ \nonumber
 &\quad +\int_{0}^t  \mathbb E\left[\trace\left((2\Gamma(x(s))-2\Gamma\left(\smat x_r(s)\\0\srix\right)+\smat 0 \\ v_1(s)\srix) K \smat 0 \\ v_1(s)\srix^\top \Sigma\right) \right] ds
 \end{align}}
for all $t\in[0, T]$.
We begin with assuming $\Sigma_2=\sigma I$, $\sigma>0$, and use this scenario for the proof of the general case. Then, we observe that
{\allowdisplaybreaks\begin{equation}\label{eins}
\begin{aligned}
\left\langle \Sigma x_-(s), \smat 0 \\ v_0(s)\srix\right\rangle &= \left\langle x_-(s), \smat 0 \\ \Sigma_2 v_0(s)\srix\right\rangle= \sigma^2 \left\langle x_+(s), \smat 0 \\  \Sigma_2^{-1} v_0(s)\srix\right\rangle\\
&=\sigma^2 \left\langle \Sigma^{-1} x_+(s), \smat 0 \\   v_0(s)\srix\right\rangle.
\end{aligned}
\end{equation}}
Further, we find that {\allowdisplaybreaks
\begin{equation}\label{zwei}
\begin{aligned}
\smat 0 \\ v_1(s)\srix^\top \Sigma = \smat 0 \\ \Sigma_2 v_1(s)\srix^\top &= \sigma^2 \smat 0 \\ \Sigma_2^{-1} v_1(s)\srix^\top
=\sigma^2 \smat 0 \\ v_1(s)\srix^\top \Sigma^{-1},\\
\trace\left(\smat 0 \\ v_1(s)\srix K \smat 0 \\ v_1(s)\srix^\top \Sigma^{-1}\right)&\leq
\trace\left(3\smat 0 \\ v_1(s)\srix K \smat 0 \\ v_1(s)\srix^\top \Sigma^{-1}\right).
\end{aligned}
\end{equation}}
Using \eqref{diagQ}, it holds that
{\allowdisplaybreaks
\begin{equation}\label{drei}
\begin{aligned}
LV_{\Sigma}(x(s), \smat x_r(s)\\0\srix; \delta)&\leq
-(h(x(s))- h_r(x_r(s)))^\top (h(x(s))- h_r(x_r(s)))\\
&=-\left\| y(s)-y_r(s)\right\|^2.
\end{aligned}
\end{equation}}
Inserting \eqref{eins}, \eqref{zwei} and \eqref{drei} into \eqref{null}, we obtain
{\allowdisplaybreaks
\begin{align*}
 \mathbb E[V_{\Sigma}(x_-(t))]&\leq  \int_{0}^t -\mathbb E\left\| y(s)-y_r(s)\right\|^2+\mathbb E[V_{\Sigma}(x_-(s))] (\left\|\tilde u(s)\right\|/\delta)^2 ds\\
 &\quad +\sigma^2\int_{0}^t \mathbb E\left[2\left\langle \Sigma^{-1} x_+(s), \smat 0 \\ v_0(s)\srix\right\rangle\right]  ds
 \\
 &\quad +\sigma^2\int_{0}^t  \mathbb E\left[\trace\left((2\Gamma(x(s))-2\Gamma\left(\smat x_r(s)\\0\srix\right)+3\smat 0 \\ v_1(s)\srix) K \smat 0 \\ v_1(s)\srix^\top \Sigma^{-1}\right) \right] ds\\
 &=\sigma^2\alpha_0(t)-\left\| y-y_r\right\|_{L^2_t}^2+\int_0^t\mathbb E[V_{\Sigma}(x_-(s))] (\left\|\tilde u(s)\right\|/\delta)^2 ds
 \end{align*}}
for all $t\in[0, T]$, where the function $\alpha_0$ is defined by $\alpha_0(t) =
 \int_{0}^t \mathbb E\left[2\left\langle \Sigma^{-1} x_+(s), \smat 0 \\ v_0(s)\srix\right\rangle\right]  ds
 +\int_{0}^t  \mathbb E\left[\trace\left((2\Gamma(x(s))-2\Gamma\left(\smat x_r(s)\\0\srix\right)+3\smat 0 \\ v_1(s)\srix) K \smat 0 \\ v_1(s)\srix^\top \Sigma^{-1}\right) \right] ds$. We can now apply Lemma \ref{gron_int} with $t_0=0$, $\alpha(t)= \sigma^2\alpha_0(t)-\left\| y-y_r\right\|_{L^2_t}^2$ and $\beta(t)= (\left\|\tilde u(t)\right\|/\delta)^2$. This provides the following estimate:
 \begin{align*}
    \mathbb E[V_{\Sigma}(x_-(t))]&\leq \sigma^2\alpha_0(t)-\left\| y-y_r\right\|_{L^2_t}^2\\
    &\quad+\int_{0}^t \Big(\sigma^2\alpha_0(s)-\left\| y-y_r\right\|_{L^2_s}^2\Big) (\left\|\tilde u(s)\right\|/\delta)^2 \exp\left\{\int_s^t (\left\|\tilde u(v)\right\|/\delta)^2 dv\right\} ds
   \end{align*}
which implies that \begin{align}\label{est_first_step}
    \left\| y-y_r\right\|_{L^2_t}^2\leq \sigma^2\alpha_0(t)
    +\int_{0}^t \sigma^2\alpha_0(s) (\left\|\tilde u(s)\right\|/\delta)^2 \exp\left\{\int_s^t (\left\|\tilde u(v)\right\|/\delta)^2 dv\right\} ds
   \end{align}
for all $t\in[0, T]$. Based on the dynamics in \eqref{eq_x_plus}, we apply Lemma \ref{lemstochdiff} once more with  $V(x)=V_{\Sigma^{-1}}(x)=x^\top \Sigma^{-1} x$, where $a(t)=f(x(t))+f\left(\smat x_r(t)\\0\srix\right)+2B u(t)+\left[G(x(t))+G\left(\smat x_r(t)\\0\srix\right)\right]u(t)-\smat 0 \\ v_0(t)\srix$ and $B(t)= \Gamma(x(t))+\Gamma\left(\smat x_r(t)\\0\srix\right)+2M(u(t))-\smat 0 \\ v_1(t)\srix$. We exploit the point symmetry of $f$, $G$ and $\Gamma$ additionally yielding $f\left(\smat x_r(t)\\0\srix\right)=-f\left(-\smat x_r(t)\\0\srix\right)$, $G\left(\smat x_r(t)\\0\srix\right)=-G\left(-\smat x_r(t)\\0\srix\right)$ and $\Gamma\left(\smat x_r(t)\\0\srix\right)=-\Gamma\left(-\smat x_r(t)\\0\srix\right)$. Analogue to \eqref{null} we hence obtain {\allowdisplaybreaks
\begin{align}\nonumber
 &\mathbb E[V_{\Sigma^{-1}}(x_+(t))]\leq  \int_{0}^t \mathbb E\left[LV_{\Sigma^{-1}}(x(s), -\smat x_r(s)\\0\srix; \delta)\right]+\mathbb E[V_{\Sigma^{-1}}(x_+(s))] (\left\|\tilde u(s)\right\|/\delta)^2 ds\\ \nonumber
 & -\int_{0}^t \mathbb E\left[2\left\langle \Sigma^{-1} x_+(s), \smat 0 \\ v_0(s)\srix\right\rangle\right]  ds+\int_{0}^t \mathbb E\left[4\left\langle \Sigma^{-1} x_+(s), B u(s)\right\rangle\right]  ds
 \\   \nonumber
 & +\int_{0}^t  \mathbb E\left[\trace\left((2\Gamma(x(s))-2\Gamma\left(-\smat x_r(s)\\0\srix\right)+2M(u(s))-\smat 0 \\ v_1(s)\srix) K \big(2M(u(s))-\smat 0 \\ v_1(s)\srix\big)^\top \Sigma^{-1}\right) \right] ds\\ \label{first_bla}
 &=  \int_{0}^t \mathbb E\left[LV_{\Sigma^{-1}}(x(s), -\smat x_r(s)\\0\srix; \delta)\right]+\mathbb E[V_{\Sigma^{-1}}(x_+(s))] (\left\|\tilde u(s)\right\|/\delta)^2 ds\\ \nonumber
 &\quad +4\int_{0}^t \mathbb E\left[\left\langle \Sigma^{-1} x_+(s), B u(s)\right\rangle+\trace\left((\Gamma(x(s))-\Gamma\left(-\smat x_r(s)\\0\srix\right)) K M(u(s))^\top \Sigma^{-1}\right)\right]  ds
 \\ \nonumber
 & \quad+4\int_{0}^t  \mathbb E\left[- \left\langle u(s), \mathcal U u(s)\right\rangle+ \trace\left(M(u(s)) K M(u(s))^\top \Sigma^{-1}\right) \right] ds+4 \int_0^t  \mathbb E\|\mathcal U^{\frac{1}{2}} u(s)\|^2 ds
 \\ \nonumber
 &\quad -\int_{0}^t  \mathbb E\left[\trace\left((2\Gamma(x(s))+2\Gamma\left(\smat x_r(s)\\0\srix\right)+4 M(u(s))-\smat 0 \\ v_1(s)\srix) K \smat 0 \\ v_1(s)\srix^\top \Sigma^{-1}\right) \right] ds\\ \nonumber
 &\quad-\int_{0}^t \mathbb E\left[2\left\langle \Sigma^{-1} x_+(s), \smat 0 \\ v_0(s)\srix\right\rangle\right]  ds
 \end{align} }
for all $t\in[0, T]$. It further holds that
{\allowdisplaybreaks
\begin{equation}\label{final_bla}
 \begin{aligned}
&\trace\left(K \smat 0 \\ v_1(s)\srix^\top \Sigma^{-1} (2\Gamma\left(\smat x_r(s)\\0\srix\right)+4 M(u(s))-\smat 0 \\ v_1(s)\srix) \right)\\
&\trace\left(K \smat 0 \\ v_1(s)\srix^\top \Sigma^{-1} (-2\Gamma\left(\smat x_r(s)\\0\srix\right)+4 \big[M(u(s))+\Gamma\left(\smat x_r(s)\\0\srix\right)\big]-\smat 0 \\ v_1(s)\srix) \right)\\
&= \trace\left(K \smat 0 \\ v_1(s)\srix^\top \Sigma^{-1} (-2\Gamma\left(\smat x_r(s)\\0\srix\right)+4\smat 0\\  \Gamma_\star(x_r(s))+M_\star(u(s))\srix -\smat 0 \\ v_1(s)\srix) \right)\\
&= \trace\left(K \smat 0 \\ v_1(s)\srix^\top \Sigma^{-1} (-2\Gamma\left(\smat x_r(s)\\0\srix\right)+3\smat 0 \\ v_1(s)\srix) \right)
\end{aligned}
\end{equation}}
using the definition of $v_1$.
We insert \eqref{final_bla} into \eqref{first_bla} and exploit \eqref{traceU_rel}, \eqref{traceS_rel} as well as the definition of $\alpha_0$ leading to {\allowdisplaybreaks
\begin{align*}
 \mathbb E[V_{\Sigma^{-1}}(x_+(t))]&\leq \int_{0}^t \mathbb E\left[LV_{\Sigma^{-1}}(x(s), -\smat x_r(s)\\0\srix; \delta)\right]+\mathbb E[V_{\Sigma^{-1}}(x_+(s))] (\left\|\tilde u(s)\right\|/\delta)^2 ds\\
 &\quad +4\int_{0}^t \mathbb E\left[\left\langle S_{\Sigma^{-1}}\left(x(s), -\smat x_r(s)\\0\srix\right),  u(s)\right\rangle - \left\langle u(s), U_{\Sigma^{-1}} u(s) \right\rangle \right]   ds\\
 &\quad +4 \int_0^t  \mathbb E\|\mathcal U^{\frac{1}{2}} u(s)\|^2 ds - \alpha_0(t)
\end{align*}}
for all $t\in [0, T]$. We set $z=-2u(s)$, $x=x(s)$ and $y=-\smat x_r(s)\\0\srix$ in \eqref{id1} and obtain {\allowdisplaybreaks
\begin{align}\nonumber
     LV_{\Sigma^{-1}}(x(s), -\smat x_r(s)\\0\srix; \delta)\leq 4 u(s)^\top U_{\Sigma^{-1}} u(s) - 4 S_{\Sigma^{-1}}\left(x(s),-\smat x_r(s)\\0\srix\right)^\top u(s).
                      \end{align}}
Therefore, we have {\allowdisplaybreaks
\begin{align*}
 \mathbb E[V_{\Sigma^{-1}}(x_+(t))]\leq  4 \|\mathcal U^{\frac{1}{2}} u\|_{L^2_t}^2 - \alpha_0(t)+ \int_{0}^t \mathbb E[V_{\Sigma^{-1}}(x_+(s))] (\left\|\tilde u(s)\right\|/\delta)^2 ds
\end{align*}}
for all $t\in[0, T]$. We apply Lemma \ref{gron_int} with $t_0=0$, $\alpha(t)= 4 \|\mathcal U^{\frac{1}{2}} u\|_{L^2_t}^2 - \alpha_0(t)$ and $\beta(t)= (\left\|\tilde u(t)\right\|/\delta)^2$. This gives us \begin{align*}
    \mathbb E[V_{\Sigma^{-1}}(x_+(t))]&\leq 4 \|\mathcal U^{\frac{1}{2}} u\|_{L^2_t}^2 - \alpha_0(t)\\
    &\quad+\int_{0}^t \left( 4 \|\mathcal U^{\frac{1}{2}} u\|_{L^2_s}^2 - \alpha_0(s)\right) (\left\|\tilde u(s)\right\|/\delta)^2 \exp\left\{\int_s^t (\left\|\tilde u(v)\right\|/\delta)^2 dv\right\} ds
   \end{align*}
for all $t\in [0, T]$ implying that \begin{align*}
  &\alpha_0(t) +\int_{0}^t \alpha_0(t) (\left\|\tilde u(s)\right\|/\delta)^2 \exp\left\{\int_s^t (\left\|\tilde u(v)\right\|/\delta)^2 dv\right\} ds  \\
  &\leq 4 \|\mathcal U^{\frac{1}{2}} u\|_{L^2_t}^2+4 \|\mathcal U^{\frac{1}{2}} u\|_{L^2_t}^2 \int_{0}^t  (\left\|\tilde u(s)\right\|/\delta)^2 \exp\left\{\int_s^t (\left\|\tilde u(v)\right\|/\delta)^2 dv\right\} ds\\
  &= 4 \|\mathcal U^{\frac{1}{2}} u\|_{L^2_t}^2 \exp\left\{\int_0^t (\left\|\tilde u(s)\right\|/\delta)^2 ds\right\}
   \end{align*}
for all $t\in [0, T]$. Combining this estimate with \eqref{est_first_step}, we see that
\begin{align}\label{estim_sigID}
 \left\| y-y_r\right\|_{L^2_t}^2\leq 4\sigma^2 \|\mathcal U^{\frac{1}{2}} u\|_{L^2_t}^2 \exp\left\{\left\|\tilde u\right\|_{L^2_t}^2/\delta^2\right\}
   \end{align}
for all $t\in [0, T]$. Below, we discuss the  case of a general $\Sigma_2$ using \eqref{estim_sigID} as a basis. We can bound the $L^2_T$-error of the dimension reduction procedure introduced in Section \ref{sec_BT} as follows \begin{align*}
  \left\|y-y_{r}\right\|_{L^2_T} 
 \leq \left\|y-y_{n-1}\right\|_{L^2_T} + \sum_{i=r+1}^{n-1} \left\| y_{k}-y_{k-1}\right\|_{L^2_T}.
 \end{align*}
 Here, $y_{k}$ denotes the output of the reduced model \eqref{red_system} with dimension $k\in\{r, \dots, n-1\}$. Applying \eqref{estim_sigID} with $t=T$, we obtain $\left\|y-y_{n-1}\right\|_{L^2_T}\leq 2\sigma_n  \|\mathcal U^{\frac{1}{2}} u\|_{L^2_T} \exp\left\{0.5\left\|\tilde u\right\|_{L^2_T}^2/\delta^2\right\}$. The error in the reduction step from $y_{n-1}$ to $y_{n-2}$ can be bounded the same way exploiting Proposition \ref{prop_red_bal}. It tells us that $y_{n-1}$ is the output of a balanced system with Gramian $\diag(\sigma_1, \dots, \sigma_{n-1})$ meaning that we can rely on inequalities like in \eqref{balanced_equations}, again. Since the truncation also preserves the point symmetry of the coefficients, we can apply \eqref{estim_sigID} once more. We continue this procedure until we arrive at the reduced system of order $r$ and find in every step that $\left\| y_{k}-y_{k-1}\right\|_{L^2_T}\leq 2\sigma_k \|\mathcal U^{\frac{1}{2}} u\|_{L^2_T} \exp\left\{0.5\left\|\tilde u\right\|_{L^2_T}^2/\delta^2\right\}$ concluding this proof.
\end{proof}
Theorem \ref{thm_error_bound} is one of the main results of this paper. It provides an a-priori error bound that gives a criterion for the choice of the reduced dimension $r$. This parameter has to be chosen, so that the sum of truncated HSVs $\sigma_{r+1}, \dots, \sigma_n$ is small. The result of Theorem \ref{thm_error_bound}  is much more general than the one in \cite{redstochbil}, where $f, G, \Gamma, h$ are linear and $M\equiv 0$. In that case, the point symmetry of $f, G, \Gamma$ is naturally given. For nonlinear systems, it seems to be crucial that the coefficients are point symmetric when algebraic Gramians associated to quadratic forms are the basis of the MOR procedure. This was essentially also noticed in \cite{gen_incr_bt}. Let us consider an example constructing a system \eqref{original_system} that satisfies the requirements of Theorem \ref{thm_error_bound}.
\begin{example}\label{ex1}
For simplicity, let us assume that $m=q=1$ and that $w$ is a scalar standard Wiener process. We set $f(x) = A x - x^{\circ 3}$ and $G(x) = N x$ given matrices $A, N\in \mathbb R^{n\times n}$. Suppose that $(A+\frac{c^2}{2}I)^\top + (A+\frac{c^2}{2}I) + \delta^2 N^\top N = -Y<0$ is a negative definite matrix for some $\delta>0$ and some   $c>0$. This is always true if  $(A+\frac{c^2}{2}I)^\top +(A+\frac{c^2}{2}I)<0$ and $\delta$ is sufficiently small. Now, let $\Gamma$ be some point symmetric Lipschitz continuous function with constant $c$, e.g., $\Gamma(x) = \sin(x)$, where the sine function is applied component-wise. Hence, we have $c=1$. Using Proposition \ref{prop_lyap}
with $X=I$ yields
 \begin{align*}
LV_I(x, y; \delta)
&=2\langle x-y, A(x-y)- (x^{\circ 3}-y^{\circ 3})\rangle + \delta^2\|N(x-y)\|^2+ \|\Gamma(x)-\Gamma(y)\|^2\\
&=\langle x-y, (A^\top+A+\delta^2 N^\top N)(x-y)\rangle-2\langle x-y,  x^{\circ 3}-y^{\circ 3}\rangle+\|\Gamma(x)-\Gamma(y)\|^2.
\end{align*}
We obtain that
\begin{align}\label{est_cubic}
 -2\langle x- y,  x^{\circ 3} - y^{\circ 3}\rangle&=-2\sum_{i=1}^n (x_i^4 +y_i^4 - y_i x_i^3-x_i y_i^3)\\ \nonumber
 &=-2\sum_{i=1}^n (x_i-y_i)^2 (x_i^2 +y_i^2 + y_i x_i)\leq -\sum_{i=1}^n (x_i-y_i)^2 (x_i +y_i)^2 \leq 0.
\end{align}
The Lipschitz continuity of $\Gamma$ yields $\|\Gamma(x)-\Gamma(y)\|^2\leq \langle x-y,  c^2 (x-y)\rangle$. Consequently, we have \begin{align*}
LV_I(x, y; \delta)
\leq \langle x-y, ((A+\frac{c^2}{2}I)^\top+(A+\frac{c^2}{2}I)+\delta^2 N^\top N)(x-y)\rangle = - \langle x-y, Y (x-y)\rangle.
\end{align*}
Therefore, we find that $LV_I(x, y; \delta)
\leq -\lambda V_I(x-y)$ for some $\lambda>0$ meaning that \eqref{exist_gram_prop} is satisfied. Now, let $h$ be an arbitrary Lipschitz continuous function. Then, all assumptions of Theorem \ref{thm_gram_exist} are satisfied and hence the Gramians $P$ and $Q$ exist. Since $f$, $G$ and $\Gamma$ are point symmetric, all requirements of Theorem \ref{thm_error_bound} are met.
\end{example}
We slightly modify Example \ref{ex1} to illustrate that $G$ can also be nonlinear and does not have to be Lipschitz continuous. However, notice that we provide a scenario, where not all assumptions of Theorem \ref{thm_error_bound} hold.
\begin{example}
Let $f$ be like in Example \ref{ex1}. We assume that $(A+\frac{c^2}{2}I)^\top + (A+\frac{c^2}{2}I) = -Y<0$ is a negative definite matrix for some $c>0$ and that $\Gamma$ is a Lipschitz map with constant $c$. We set $G(x)=x^{\circ 2}$ and obtain \begin{align*}
LV_I(x, y; \delta)
\leq \langle x-y, ((A+\frac{c^2}{2}I)^\top+(A+\frac{c^2}{2}I))(x-y)\rangle -2\langle x- y,  x^{\circ 3} - y^{\circ 3}\rangle + \delta^2\|x^{\circ 2}-y^{\circ 2}\|^2.
\end{align*}
We obtain that \begin{align*}
\delta^2\|x^{\circ 2}-y^{\circ 2}\|^2 =\delta^2\sum_{i=1}^n(x_i^2-y_i^2)^2= \delta^2\sum_{i=1}^n(x_i-y_i)^2(x_i+y_i)^2
               \end{align*}
and consequently, using \eqref{est_cubic}, it holds that \begin{align*}
LV_I(x, y; \delta)
\leq \langle x-y, ((A+\frac{c^2}{2}I)^\top+(A+\frac{c^2}{2}I))(x-y)\rangle =- \langle x-y, Y(x-y)\rangle \leq -\lambda V_I(x-y)
\end{align*}
for some $\lambda >0$ given that $0<\delta\leq 1$. Therefore, the assumptions of Theorem \ref{thm_gram_exist} are satisfied if $h$ is Lipschitz continuous. For that reason the Gramians $P$ and $Q$ exist. However, the error bound of Theorem \ref{thm_error_bound} might not hold as we are lacking point symmetry.
\end{example}

We conclude this paper with formulating several auxiliary results that we frequently use within this work.

\appendix
\section{Supporting lemmas}
Below, let us state Ito's formula that is frequently used in this manuscript.
\begin{lem}\label{lemstochdiff}
Suppose that $a, b_1, \ldots, b_q$ are $\mathbb R^n$-valued $(\mathcal F_t)_{t\in[0, T]}$-adapted processes with $a$ being almost surely Lebesgue integrable and $b_i$ satisfying $\int_0^T \|b_i(t)\|^2 dt<\infty$ almost surely. Let
 $w=\begin{bmatrix} w_1& \ldots & w_q\end{bmatrix}^\top$ be a Wiener  process with covariance matrix $K=(k_{ij})$. If $x$ is an Ito process given by \begin{align*}
 dx(t)=a(t) dt+ B(t)dw(t),                                                                                                                                  \end{align*}
where $B=\begin{bmatrix} b_1& \ldots & b_q\end{bmatrix}$. Then, given a sufficiently smooth function $V:\mathbb R^n\to \mathbb R$, we have \begin{align*}
 dV(x(t))=&\left[\left\langle D V(x(t)), a(t)\right\rangle + \frac{1}{2}  \trace\left(B(t) K B(t)^\top D^2 V(x(t))\right) \right] dt\\
 &+    \left\langle D V(x(t)), B(t)dw(t)\right\rangle                                                                                                                       \end{align*}
\end{lem}
\begin{proof}
This result is well-known and a proof can for instance be found in \cite[Theorem 6.4]{mao}.
\end{proof}
We need various versions of Gronwall's lemma in this paper. In order to render this paper self-contained, we prove each version that we state below. We begin with a differential form that is essential in the context of analyzing stability of (stochastic) systems.
\begin{lem}[Gronwall's lemma -- differential form]\label{gron_dif}
Let $I$ be either $[0, \infty)$ or $[0, T]$ for some $T>0$. Suppose that $z: I\rightarrow \mathbb R$ is absolutely continuous and $\beta:I \rightarrow \mathbb R$ is (locally) integrable on $I$. If \begin{align*}
   \dot z(t)\leq  \beta(t) z(t)
   \end{align*}
holds Lebesgue almost everywhere in $I$, then we have\begin{align*}
 z(t) \leq z(0) \expn^{\int_0^t\beta(s)ds}
\end{align*}
for all $t\in I$.
\end{lem}
\begin{proof}
 We define $\tilde z(t):= z(t) \expn^{-\int_0^t\beta(s)ds}$ and obtain $\dot{\tilde z}(t) =\expn^{-\int_0^t\beta(s)ds}(\dot z(t)-\beta(t) z(t))\leq 0$ for almost all $t\in I$. Integrating this inequality yields $ z(t) \expn^{-\int_0^t\beta(s)ds}-z(0)=\tilde z(t)-\tilde z(0)\leq 0$ for all $t\in I$.
\end{proof}
The following integral version of Gronwall's lemma is vital for the error analysis and the dominant subspace detection  in this paper.
\begin{lem}[Gronwall's lemma -- integral form]\label{gron_int}
Let $I$ be either $[t_0, \infty)$ or $[t_0, T]$ for some $T>0$ and $0\leq t_0< T$. Suppose that $z, \alpha: I \rightarrow \mathbb R$ are measurable (locally) bounded functions and $\beta: I\rightarrow \mathbb R$ is a nonnegative integrable function.
If \begin{align}\label{assum_gron}
    z(t)\leq \alpha(t)+\int_{t_0}^t \beta(s) z(s) ds
   \end{align}
for all $t\in I$, then it holds that \begin{align}\label{gronwallineq}
    z(t)\leq \alpha(t)+\int_{t_0}^t \alpha(s)\beta(s) \exp\left\{\int_s^t \beta(v)dv\right\} ds
   \end{align}
for all $t\in I$. If $\alpha$ is absolutely continuous, then \eqref{gronwallineq} yields \begin{align}\label{gronwallineq_abs_cont}
    z(t)\leq \exp\left\{\int_{t_0}^t \beta(v)dv\right\} \alpha(t_0)+\int_0^t \dot \alpha(s) \exp\left\{\int_s^t \beta(v)dv\right\} ds,
   \end{align}
where $\dot \alpha$ is the derivative of $\alpha$ Lebesgue almost everywhere. Moreover, a non-decreasing and continuous function $\alpha$ implies that \eqref{gronwallineq} becomes \begin{align*}
    z(t)\leq \alpha(t)\exp\left\{\int_{t_0}^t \beta(s)ds\right\}
   \end{align*}
   for all $t\in I$.
\end{lem}
\begin{proof}
By assumptions, all appearing integrals are well-defined. We define the function $\tilde z(t)= \exp\left\{-\int_{t_0}^t \beta(v)dv\right\} \int_{t_0}^t \beta(s) z(s) ds$. It is differentiable Lebesgue almost everywhere. Therefore, we find that \begin{align*}
 \dot{\tilde z}(t) = \exp\left\{-\int_{t_0}^t \beta(v)dv\right\} \beta(t)\Big(z(t)-\int_{t_0}^t \beta(s) z(s) ds\Big)\leq \exp\left\{-\int_{t_0}^t \beta(v)dv\right\} \beta(t) \alpha(t)                                                                                                                                                                                                                                      \end{align*}
for almost all $t\in I$ using assumption \eqref{assum_gron} and  $\beta$ being nonnegative. We integrate this inequality over $[t_0, t]$ and obtain \begin{align*}
 \tilde z(t)\leq \int_{t_0}^t  \exp\left\{-\int_{t_0}^s \beta(v)dv\right\} \beta(s) \alpha(s) ds                                                                                                                                           \end{align*}
for all $t\in I$. Again, by \eqref{assum_gron}, we observe that \begin{align*}
\exp\left\{-\int_{t_0}^t \beta(v)dv\right\}(z(t)-\alpha(t))&\leq \exp\left\{-\int_{t_0}^t \beta(v)dv\right\} \int_{t_0}^t \beta(s) z(s) ds = \tilde z(t)\\
&\leq \int_{t_0}^t  \exp\left\{-\int_{t_0}^s \beta(v)dv\right\} \beta(s) \alpha(s) ds
                    \end{align*}
for all $t\in I$. This leads to \eqref{gronwallineq}. If $\alpha$ is absolutely continuous, we can use integration by parts leading \begin{align*}
&\int_{t_0}^t \alpha(s)\beta(s) \exp\left\{\int_s^t \beta(v)dv\right\} ds \\
&= \Big[ -\alpha(s)\exp\left\{\int_s^t \beta(v)dv\right\}\Big]_{t_0}^t + \int_{t_0}^t \dot\alpha(s) \exp\left\{\int_s^t \beta(v)dv\right\} ds.                                                                                                                             \end{align*}
This gives us \eqref{gronwallineq_abs_cont}. Given that $\alpha$ is non-decreasing and continuous, we obtain from \eqref{gronwallineq} that \begin{align*}
    z(t)&\leq \alpha(t)\Big(1+\int_{t_0}^t \beta(s) \exp\left\{\int_s^t \beta(v)dv\right\} ds \Big)= \alpha(t)\Big(1+\Big[ -\exp\left\{\int_s^t \beta(v)dv\right\}\Big]_{t_0}^t\Big)\\
    &= \alpha(t)\exp\left\{\int_{t_0}^t \beta(v)dv\right\}
   \end{align*}
for all $t\in I$. This concludes the proof.
\end{proof}

\section*{Acknowledgments}
 MR is supported by the DFG via the individual grant ``Low-order approximations for large-scale problems arising in the context of high-dimensional
PDEs and spatially discretized SPDEs''-- project number 499366908.

\bibliographystyle{abbrv}
\bibliography{ref_nonlinear_mor}

\end{document}